\documentclass[a4paper,american,reqno]{amsart}

\newif\ifPreprint \Preprinttrue
\newif\ifSubmission \Submissionfalse

\usepackage{babel}
\usepackage[utf8]{inputenc}
\usepackage[T1]{fontenc}
\usepackage{csquotes}
\usepackage[binary-units=true]{siunitx}
\usepackage{url}
\usepackage{xspace}
\usepackage{relsize}
\usepackage{enumitem}
\usepackage{algorithm,algorithmic}
\usepackage{todonotes}
\usepackage{csquotes}
\usepackage{microtype}
\usepackage{accents}
\usepackage{notationForne}
\usepackage[style = numeric-comp,
            maxbibnames = 100,
            maxcitenames = 2,
            firstinits=true,
            isbn=false,
            backend = bibtex]{biblatex}
\usepackage{longtable}
\usepackage{booktabs}
\usepackage{hyperref} 

\makeatletter
\patchcmd{\@settitle}{\uppercasenonmath\@title}{\scshape\large}{}{}
\patchcmd{\@setauthors}{\MakeUppercase}{\scshape\normalsize}{}{}
\makeatother

\vfuzz \hfuzz
\newtheorem{theorem}{Theorem}[section]
\newtheorem{lemma}[theorem]{Lemma}

\newtheorem{assumption}{Assumption}

\theoremstyle{definition}

\theoremstyle{remark}

\bibliography{padm-feas-pump}

\begin{document}

\title[Penalty ADMs for Mixed-Integer Optimization]%
{Penal\/ty Alternating Direction Methods\\for Mixed-Integer
  Optimization:\\A New View on Feasibility Pumps}
\author[B. Gei{\ss}ler, A. Morsi, L. Schewe, M. Schmidt]
{Björn Gei{\ss}ler$^1$, Antonio Morsi$^1$, Lars Schewe$^1$, Martin
  Schmidt$^2$}
\address{$^1$Björn Gei{\ss}ler, Antonio Morsi, Lars Schewe,
  Friedrich-Alexander-Universit\"at Erlangen-Nürn\-berg (FAU),
  Discrete Optimization,
  Cauerstr.~11, 91058 Erlangen, Germany\\
  $^2$Martin Schmidt,
  (a)
  Friedrich-Alexander-Universit\"at Erlangen-Nürn\-berg (FAU),
  Discrete Optimization,
  Cauerstr.~11, 91058 Erlangen, Germany;
  (b) Energie Campus N\"urnberg,
  F\"urther Str.~250, 90429 N\"urnberg, Germany}
\email{$^1$\{bjoern.geissler,antonio.morsi,lars.schewe\}@math.uni-erlangen.de}
\email{$^2$mar.schmidt@fau.de}

\date{\today}

\begin{abstract}
  Feasibility pumps are highly effective primal heuristics for
\mixedinteger linear and \nonlinear optimization.
However, despite their success in practice there are only few works
considering their theoretical properties.
We show that feasibility pumps can be seen as alternating
direction methods applied to special reformulations of the original
problem, inheriting the convergence theory of these methods.
Moreover, we propose a novel penalty framework that encompasses
this alternating direction method, which allows us to refrain from random
perturbations that are applied in standard versions of feasibility
pumps in case of failure.
We present a convergence theory for the new penalty based alternating
direction method and compare the new variant of the feasibility
pump with existing versions in an extensive numerical study for
\mixedinteger linear and \nonlinear problems.


\end{abstract}

\keywords{Mixed-Integer Nonlinear Optimization,
  Mixed-Integer Linear Optimization,
  Feasibility Pump,
  Alternating Direction Methods,
  Penalty Methods}

\subjclass[2010]{%
  65K05, 
  90-08, 
  90C10, 
  90C11, 
  90C59
}

\maketitle

Due to their practical relevance, \mixedinteger \nonlinear problems
(\minlps) form a very important class of optimization problems.
One important part of successful algorithms for the solution of such problems
is finding feasible solutions quickly.
For this, typically heuristics are employed.
These can be roughly divided into heuristics that improve known feasible solutions (\eg local
branching~\cite{Fischetti_Lodi:2003} or RINS~\cite{Danna_et_al:2005}) and
heuristics that construct feasible solutions from scratch.
This article discusses a heuristic of the latter type:
The algorithm of interest in this article is the
\socalled \emph{feasibility pump} that has originally been proposed
\ifPreprint
by~\citeauthor{Fischetti_et_al:2005}
\fi
in~\cite{Fischetti_et_al:2005}
for \mips and that has been extended by many other researchers, \eg
in~\cite{Achterberg_Berthold:2007,%
Baena_Castro:2011,%
Bertacco_et_al:2007,%
Boland_et_al:2012,%
Boland_et_al:2014,%
DeSantis_et_al:2013,%
DeSantis_et_al:2014,%
Fischetti_Salvagnin:2009,%
Hanafi_et_al:2010,%
dey2016improving}.
In addition, feasibility pumps have also been applied to \minlps
during the last years; \cf, \eg~\cite{Bonami_et_al:2009,%
Berthold:2014,%
Bonami_Goncalves:2012,%
DAmbrosio_et_al:2012,%
DAmbrosio_et_al:2010,%
Sharma:2013,%
Sharma_et_al:2015}.
A more detailed review of the literature about feasibility pumps is
given in \refsec{sec:feasibility-pumps}.
For a comprehensive overview over primal heuristics for \mixedinteger
linear and \nonlinear problems in general, we refer the interested reader to
\textcite{Berthold:2006,Berthold:2014} and the references
therein.

In a nutshell, feasibility pumps work as follows:
given an optimal solution of the continuous relaxation of the problem,
the methods construct two sequences.
The first one contains integer-feasible points, the second one
contains points that are feasible \wrt the continuous relaxation.
Thus, one has found an overall feasible point if these sequences converge to
a common point.
To escape from situations where the construction of the sequences gets
stuck and thus do not converge to a common point, feasibility pumps usually
incorporate randomized restarts.

The feasibility pumps described in the literature are difficult to
analyze theoretically due to the use of random perturbations.
These random perturbations are, however, crucial to the practical performance
of the methods.
The main object of the existing theoretical analysis is
the \emph{idealized
feasibility pump}, \ie the method without random perturbations.
This is the method analyzed in the publications~\cite{DeSantis_et_al:2013}
and~\cite{Boland_et_al:2012}.
To be more specific,
\ifPreprint
\citeauthor{DeSantis_et_al:2013} show in~\cite{DeSantis_et_al:2013}
\fi
that idealized feasibility pumps are a
special case of the Frank--Wolfe algorithm applied to a suitable chosen concave
and \nonsmooth merit function
\ifPreprint
and~\citeauthor{Boland_et_al:2012} showed in~\cite{Boland_et_al:2012}
\fi
that idealized feasibility pumps can be
seen as discrete versions of the proximal point algorithm.
However, the analysis presented in both mentioned publications cannot
be applied to feasibility pumps that use random perturbations.
Recently, there has been progress in the understanding of
the randomization step.
In \cite{dey2016improving}, the authors show that changing the randomization
step can yield a method that, at least for certain instances, is guaranteed
to produce feasible solutions with high probability.

Our approach is similar to these publications. We show
that the idealized variant can be seen as an
\emph{alternating direction method} (\adm) applied to a special
reformulation of the \mixedinteger problem at hand.
To this end, we extend the known theory on feasibility pumps by
applying the convergence theory of general \adms.
We then go one step further:
The necessity to use random perturbations comes
from the need to escape from undesired points.
We replace these random perturbations of the original feasibility pump
by a penalty framework.
This allows us to view feasibility pumps as penalty based alternating
direction methods---a new class of optimization methods for which we
also present convergence theory.
In summary, we are able to give a convergence theory for a class of feasibility
pumps that incorporates deterministic restart rules.
Another advantage is that our method can be presented in a quite
generic way that comprises both the case of \mixedinteger linear and
\nonlinear problems.

We further give extensive computational results to show that our
replacement of the random restarts also works well in practice.
In particular, our method compares favorably with published variants
of feasibility pumps for \mips and \minlps\ \wrt solution quality.

The paper is organized as follows:
In \refsec{sec:feasibility-pumps} we review the main ingredients of
feasibility pumps and give a more detailed literature survey.
Afterward, we discuss general \adms in \refsec{sec:adm} and show
that idealized feasibility pumps can be seen as \adms applied to
certain equivalent reformulations of the original problem.
In \refsec{sec:padm}, we then present a penalty ADM, prove convergence
results, and show how this new method can be used to obtain
a novel feasibility pump algorithm that replaces
random restarts with penalty parameter updates.
In \refsec{sec:implementation-issues} we discuss important
implementation issues and \refsec{sec:comp-results} finally presents
an extensive computational study both for \mips and \minlps.%

\section{Feasibility Pumps}
\label{sec:feasibility-pumps}

In this section we give an overview over feasibility pump
algorithms for \mixedinteger linear problems (\mips) as well as for
\mixedinteger \nonlinear problems (\minlps).
We start with the \mip case in \refsec{sec:feas-pumps-for-mips} and
afterward discuss generalizations for \nonlinear problems in
\refsec{sec:feas-pumps-for-minlps}.
General surveys on this topic can be found in \textcite{Berthold:2014}
and \textcite{Bonami_et_al:2012}.
\subsection{Feasibility Pumps for \MixedInteger Linear Problems}
\label{sec:feas-pumps-for-mips}

The feasibility pump has been introduced
\ifPreprint
by~\textcite{Fischetti_et_al:2005}
\fi
for binary \mips and has been
extended to general \mips
\ifPreprint
by~\textcite{Bertacco_et_al:2007}.
\fi
The goal of the feasibility pump is to find a feasible
point of a \mixedinteger linear problem of the general form
\begin{subequations}
  \label{eq:MIP}
  \begin{align}
    \min_x \quad & \transpose{c} x \label{eq:MIP:objective}\\
    \st \quad & Ax \geq b,\label{eq:MIP:linear_constraints}\\
                 & x_\intIdx \in \integers
                   \qConstraintFor{\intIdx \in \intIdxSet},
                   \label{eq:MIP:integer_constraints}
  \end{align}
\end{subequations}
where $c \in \reals^n$, $b \in \reals^m$, $A \in \reals^{m \times n}$,
and $\emptyset \neq \intIdxSet \subseteq \set{1, \dotsc, n}$.
Moreover, we assume that variable bounds
$\lb \leq x \leq \ub$ with $-\infty < \lb_\intIdx \leq
\ub_\intIdx < \infty$ for all $\intIdx \in \intIdxSet$ are part of $Ax
\geq b$.
We refer to the polyhedron of the LP relaxation of~\eqref{eq:MIP} by
$\polyhedron$, \ie $\polyhedron = \condset{x \in \reals^n}{Ax \geq b}$.
Throughout the paper we assume that $\polyhedron \neq \emptyset$.
The main idea of the feasibility pump is to create two sequences
%
$(\bar{x}^k)$ and $(\tilde{x}^k)$ such that $\bar{x}^k \in
\polyhedron$ and $\tilde{x}^k$ is integer feasible, \ie $\tilde{x}^k_\intIdx
\in \integers$ for all $\intIdx \in \intIdxSet$.
In addition, the construction of the sequences is tailored to minimize
the distance of the pairs $\bar{x}^k$ and $\tilde{x}^k$.
The algorithm terminates after a given time, after an iteration limit
has been reached, or if the distance is zero, \ie $\bar{x}^k = \tilde{x}^k$.
In the latter case, the algorithm terminates with an \mip-feasible
point, \ie a point that is both in $\polyhedron$ and integer
feasible.
We now describe the method in detail for the case of $0$-$1$-\mips,
\ie we replace $x_\intIdx \in \integers$ by $x_\intIdx \in \set{0,1}$
in~\eqref{eq:MIP:integer_constraints}.
The initial point~$\bar{x}^0$ is computed to be an optimal solution of
the LP relaxation of~\eqref{eq:MIP}, \ie $\bar{x}^0 \in \argmin
\condset{\transpose{c} x}{x \in \polyhedron}$, and the initial
point~$\tilde{x}^0$ of
the other sequence is the rounding~$\rounding{\bar{x}^0}$ of
the integer components of $\bar{x}^0$.
Note that the rounding operator $\rounding{\cdot}$ only rounds
integer components, \ie
\begin{equation*}
  \rounding{x_\intIdx} \define
  \begin{cases}
    \lfloor x_\intIdx + 0.5 \rfloor , & \text{if } \intIdx \in \intIdxSet,
    \\
    x_\intIdx, & \text{otherwise}.
  \end{cases}
\end{equation*}
From then on, in each iteration~$k$ the new iterate $\bar{x}^{k+1}$ is
the nearest point (\wrt the integer components) to $\tilde{x}^k$ in
the $\ell_1$~norm, \ie
\begin{equation*}
  \bar{x}^{k+1} \in \argmin
  \condset{\norm{x_\intIdxSet-\tilde{x}_\intIdxSet^k}_1}{x \in \polyhedron}
\end{equation*}
and $\tilde{x}^{k+1} \define \rounding{\bar{x}^{k+1}}$.
Here and in what follows, $x_\intIdxSet$ denotes the sub-vector of $x$
only consisting of the components indicated by the index
set~$\intIdxSet$.
After every rounding step a cycle and a stalling test
decides whether a random perturbation of the integer part of
$\tilde{x}^k$ is applied.
The details can be found in
\textcite{Fischetti_et_al:2005,Bertacco_et_al:2007}.
A formal listing of the basic feasibility pump for binary \mips is
given in~\refalgo{alg:basic-fp-0-1-mips}.
\begin{algorithm}[tp]
  \caption{The basic feasibility pump for $0$-$1$-MIPs}
  \label{alg:basic-fp-0-1-mips}
  \begin{algorithmic}[1]
    \STATE{Compute $\bar{x}^0 \in \argmin \condset{\transpose{c}
        x}{x \in \polyhedron}$.}
    \IF{$\bar{x}^0$ is integer feasible}
    \RETURN $\bar{x}^0$
    \ENDIF
    \STATE{Set $\tilde{x}^0 = \rounding{\bar{x}^0}$ and $k \leftarrow 0$.}
    \WHILE{not termination condition}
    \STATE{Compute $\bar{x}^{k+1} \in \argmin \condset{\norm{x_\intIdxSet -
          \tilde{x}_\intIdxSet^{k}}_1}{x \in
        \polyhedron}$.\label{alg:basic-fp-0-1-mips:projection-step}}
    \IF{$\bar{x}^{k+1}$ is integer feasible}
    \RETURN $\bar{x}^{k+1}$
    \ENDIF
    \STATE{Set $\tilde{x}^{k+1} =
      \rounding{\bar{x}^{k+1}}$.\label{alg:basic-fp-0-1-mips:rounding-step}}
    \IF{algorithm stalls or cycles}
    \STATE{perturb $\tilde{x}^{k+1}$\label{alg:basic-fp-0-1-mips:perturb-step}}
    \ENDIF
    \STATE{Set $k \ot k+1$.}
    \ENDWHILE
  \end{algorithmic}
\end{algorithm}%
In what follows, Line~\ref{alg:basic-fp-0-1-mips:projection-step} of
\refalgo{alg:basic-fp-0-1-mips} is referred to as the \emph{projection
  step} and Line~\ref{alg:basic-fp-0-1-mips:rounding-step} is called
the \emph{rounding step}.
Note that the projection step can be written as a linear program by
reformulating the $\ell_1$~norm objective as
\begin{equation*}
  \norm{x_\intIdxSet - \tilde{x}_\intIdxSet}_1
  = \sum_{\intIdx \in \intIdxSet: \tilde{x}_\intIdx = 0} x_\intIdx
  + \sum_{\intIdx \in \intIdxSet: \tilde{x}_\intIdx = 1}
  (1-x_\intIdx).
\end{equation*}
This is also possible for general \mips but then requires
the introduction of auxiliary variables and constraints;
\cf~\textcite{Bertacco_et_al:2007} for the details.
The original feasibility pump is very successful in
quickly finding feasible solutions.

However, these solutions are often of minor quality.
Thus, improvements of the original version with the goal of developing
variants of the feasibility pump that are comparable in running time as well
as success rate and provide solutions of better quality are studied in
many publications.
\ifPreprint
\textcite{Fischetti_Salvagnin:2009}
\fi
improved the rounding step by
replacing the simple rounding $\tilde{x}^k = \rounding{\bar{x}^k}$ with a
procedure based on constraint propagation.
Other strategies of improving the rounding step are given in
\textcite{Baena_Castro:2011}, where simple rounding is replaced with
rounding of different candidate points on a line segment between the
solution of the projection step and the analytic center of the LP
polyhedron, and in \textcite{Boland_et_al:2014}, where an integer line
search is applied to find integer feasible points that are closer
to~$\polyhedron$ than the points achieved by simple rounding.
In contrast to these approaches, \textcite{Achterberg_Berthold:2007}
improved the projection step in order to achieve feasible solutions of
better quality by replacing the $\ell_1$~norm
objective~$\norm{x_\intIdxSet - \tilde{x}_\intIdxSet}_1$ with a convex
combination of this distance measure and the original objective
function:
\begin{equation*}
  (1 - \alpha) \norm{x_\intIdxSet - \tilde{x}_\intIdxSet}_1
  + \alpha \beta \transpose{c} x.
\end{equation*}
Here, $\alpha \in [0,1]$ is the convex combination parameter and
$\beta \in \reals_{>0}$ is a problem data depending scaling parameter;
\cf~\cite{Achterberg_Berthold:2007} for the details.

None of the papers cited so far contains any theoretical results on
the feasibility pump.
This situation changed with a remark in
\textcite{Eckstein_Nediak:2007} noticing that the idealized
feasibility pump for binary \mips can be interpreted as a Frank--Wolfe
algorithm (\cf~\textcite{Frank_Wolfe:1956}) applied to the
minimization of a concave and \nonsmooth objective function over a
polyhedron.
\ifPreprint
\textcite{DeSantis_et_al:2013}
\fi
seized this idea and proved this
correspondence, yielding the first theoretical result for the
feasibility pump.
To be more precise, they used the result shown
\ifPreprint
by~\citeauthor{Mangasarian:1997} in~\cite{Mangasarian:1997}
\fi
that the Frank--Wolfe algorithm applied to the above given situation
terminates after a finite number of iterations and returns a \socalled
\emph{vertex stationary point} of the problem.
We discuss the relation of this result with our results in more detail
in \refsec{sec:fps-as-adms}.
In~\cite{DeSantis_et_al:2014},
\ifPreprint
\citeauthor{DeSantis_et_al:2014}
\fi
generalized their results to general \mips.
Another theoretical result is presented
\ifPreprint
by~\citeauthor{Boland_et_al:2012}
\fi
in~\cite{Boland_et_al:2012}, where
it is shown that the idealized feasibility pump can be interpreted as a
discrete version of the proximal point algorithm.
We remark that both cited theoretical investigations only consider the
case of idealized feasibility pumps, \ie the variants of feasibility
pumps without random perturbations used to handle cycling or stalling
issues.


\subsection{Feasibility Pumps for \MixedInteger \Nonlinear Problems}
\label{sec:feas-pumps-for-minlps}

We now turn to feasibility pumps for \mixedinteger
\nonlinear problems of the form
\begin{subequations}
  \label{eq:MINLP}
  \begin{align}
    \min_x \quad & \objFun(x)\label{eq:MINLP:objective}\\
    \st \quad & \ineqCons(x) \geq 0,\label{eq:MINLP:general-constraints}\\
                 & x_\intIdx \in \integers \cap [l_\intIdx, u_\intIdx]
                   \qConstraintFor{\intIdx \in \intIdxSet},
                   \label{eq:MINLP:integer-constraints}
  \end{align}
\end{subequations}
where the objective function~$\objFun : \reals^n \to \reals$ and the
constraints function~$\ineqCons : \reals^n \to \reals^\nrIneqCons$ are continuous.
The feasible set of the \nlp relaxation is denoted by
$\feasSetNLPRelax \define \condset{x \in \reals^n}{\ineqCons(x) \geq 0}$.
Again, we assume that $\feasSetNLPRelax \neq \emptyset$ holds and that
$\feasSetNLPRelax$ is compact.
Lastly, we assume that all bounds on the discrete variables are
finite, \ie $- \infty < l_\intIdx \leq u_\intIdx < \infty$ for all
$\intIdx \in \intIdxSet$.
The \minlp~\eqref{eq:MINLP} is said to be convex if $\objFun$ and $-\ineqCons$ are
convex.

For convex problems, the direct generalization of the feasibility pump
for \mips to \minlps is given in \textcite{Bonami_Goncalves:2012}:
the projection step \lp is replaced by an \nlp and the rounding step stays the
same, \ie $\tilde{x}^{k+1} \define \rounding{\bar{x}^{k+1}}$.
Cycling and stalling issues are again handled by random perturbations
of the integer components of~$x$.
Variants of this method are then deduced by modifying the rounding
step and by replacing the $\ell_1$ with the $\ell_2$~norm in the
projection step.
The feasibility pump for convex \minlp proposed
in~\textcite{Bonami_et_al:2009} significantly differs from the version
in~\cite{Bonami_Goncalves:2012} because it replaces the simple rounding
step by a \mip relaxation of~\eqref{eq:MINLP} that is successively
tightened by adding outer approximation
cuts (\cf~\textcite{Duran_Grossmann:1986}) based on the \nlp feasible
solutions obtained by solving the $\ell_2$~norm projection steps.
\ifPreprint
\citeauthor{Bonami_et_al:2009} study two versions of their algorithm;
\fi
a basic and an enhanced one.
For their basic version it is shown that it cannot cycle if the LICQ
holds for the \nlp relaxation.
For their enhanced version it is shown that it cannot cycle and that it
is an exact method for solving convex \minlps if all integer variables
are bounded.

For \nonconvex \minlps the convergence results of
\textcite{Bonami_et_al:2009} do not hold because the outer
approximation cuts cannot be applied to the \nonconvex \nlp relaxation
and because the projection step is now a \nonconvex problem, which
is too hard to be solved to global optimality in general.
The article \textcite{DAmbrosio_et_al:2010} is the first presentation of a feasibility
pumps for \nonconvex \minlps.
The above mentioned issues are ``resolved'' by solving the \nonconvex
projection step \nlp via a multistart heuristic using local \nlp
solvers and the rounding step is realized by a \mip using outer
approximation cuts if they are globally valid for the \nonconvex
problem.
In the subsequent paper~\cite{DAmbrosio_et_al:2012},
\ifPreprint
\citeauthor{DAmbrosio_et_al:2012}
\fi
interpreted feasibility pumps for
general \nonconvex \minlps as variants of successive projection
methods (SPMs).
Despite their strong similarity, the authors observe that typical
feasibility pumps do not fall exactly into the class of SPMs, which is
why their convergence theory is not applicable.%

Finally, a generalization of the objective feasibility pump of
\textcite{Achterberg_Berthold:2007} is given in
\cite{Sharma:2013,Sharma_et_al:2015} for convex \minlp and
\ifPreprint
\textcite{Berthold:2014}
\fi
discusses some new algorithmic ideas for
\nonconvex \minlps. 



\section{Alternating Direction Methods}
\label{sec:adm}

In this section, we first brief\/ly review classical alternating
direction methods (ADMs) and afterward prove that idealized
feasibility pumps, \ie the basic feasibility pump
(Algorithm~\ref{alg:basic-fp-0-1-mips}) without random perturbations,
can be seen as a special case of alternating direction methods.
This gives new theoretical insights since the complete theory of ADMs
can be applied to idealized feasibility pumps.

To this end, we consider the general problem
\begin{subequations}
  \label{eq:abstract-problem}
  \begin{align}
    \min_{x,y} \quad & \objFunADM(x,y)\\
    \st \quad & \eqCons(x,y) = 0, \quad \ineqCons(x,y) \geq 0,\\
                     & x \in X, \quad y \in Y,
  \end{align}
\end{subequations}
for which we make the following assumption:
\begin{assumption}
  \label{ass:general}
  The objective function~$\objFunADM : \reals^{n_x + n_y} \to \reals$ and the constraint
  functions~$\eqCons : \reals^{n_x + n_y} \to \reals^\nrEqCons$,
  $\ineqCons : \reals^{n_x + n_y} \to \reals^\nrIneqCons$ are
  continuous and the sets~$X$ and $Y$ are
  non-empty and compact.
\end{assumption}
The feasible set is denoted by $\feasSet$, \ie
\begin{equation*}
  \feasSet
  =
  \condset{(x,y) \in X \times Y}{\eqCons(x,y) = 0, \ \ineqCons(x,y) \geq 0}
  \subseteq
  X \times Y,
\end{equation*}
and the corresponding projections onto $X$ and $Y$  are
denoted by $\feasSet_X$ and $\feasSet_Y$, respectively.
Classical alternating direction methods are extensions of Lagrangian
methods and have been originally proposed in
\textcite{Gabay_Mercier:1976} and \textcite{Glowinski_Marroco:1975}.
More recently, \adm-type methods have seen a resurgence; see, \eg \cite{Boyd_et_al:2011} for a general overview
and \cite{Geissler_et_al:2015} for an application of \adms to
nonconvex \minlps from gas transport including heat power constraints.
The latter application also provided the motivation
for this article.
ADMs solve Problem~\eqref{eq:abstract-problem} by
solving two simpler problems:
Given an iterate~$(x^k, y^k)$ they solve
Problem~\eqref{eq:abstract-problem} for $y$ fixed to $y^k$ into the
direction of~$x$, yielding a new $x$-iterate~$x^{k+1}$.
Afterward, $x$ is fixed to $x^{k+1}$ and
Problem~\eqref{eq:abstract-problem} is solved into the direction
of~$y$, yielding a new $y$-iterate~$y^{k+1}$.
A formal listing is given in \refalgo{alg:std-adm}.
Note that we do not state a practical termination criterion in
\refalgo{alg:std-adm} in order to facilitate a more streamlined
analysis.
For the implementation details we refer to \refsec{sec:implementation-issues}.
\begin{algorithm}[tp]
  \caption{A Standard Alternating Direction Method}
  \label{alg:std-adm}
  \begin{algorithmic}[1]
    \STATE{Choose initial values $(x^0, y^0) \in X \times Y$.}
    \FOR{$k=0,1,\dotsc$}%
    \STATE{Compute
      \begin{equation*}
        x^{k+1} \in \argmin_x \condset{\objFunADM(x,y^k)}{\eqCons(x,y^k)
          = 0, \ \ineqCons(x,y^k) \geq 0, \ x \in X}.
      \end{equation*}\label{alg:std-adm-problem-1}}
      \vspace*{-1em}
      \STATE{Compute
        \begin{equation*}
          y^{k+1} \in \argmin_y \condset{\objFunADM(x^{k+1},y)}{\eqCons(x^{k+1},y)
            = 0, \ \ineqCons(x^{k+1},y) \geq 0, \ y \in Y}.
        \end{equation*}
        \label{alg:std-adm-problem-2}}
      \vspace*{-1em}
      \STATE{Set $k \ot k+1$.}
    \ENDFOR
  \end{algorithmic}
\end{algorithm}%

If the optimization problem in Line~\ref{alg:std-adm-problem-1} or
Line~\ref{alg:std-adm-problem-2} of \refalgo{alg:std-adm} has a unique
solution for all $k$, it is known that ADMs converge to \socalled
\emph{partial minima} of Problem~\eqref{eq:abstract-problem}, \ie to
points~$(x^*,y^*) \in \feasSet$ for which
\begin{align*}
  \objFunADM(x^{*},y^{*}) \leq \objFunADM(x,y^{*})
  & \quad
    \text{for all } (x,y^*) \in \feasSet,
  \\
  \objFunADM(x^{*}, y^{*}) \leq \objFunADM(x^{*}, y)
  & \quad
    \text{for all } (x^*,y) \in \feasSet
\end{align*}
holds; \cf \textcite{Gorski_et_al:2007} for the following result:
\begin{theorem}
  Let ${\left\{ (x^{i},y^i) \right\}}_{i=0}^{\infty}$ be a sequence with
  $(x^{i+1},y^{i+1}) \in \Theta (x^i,y^i)$, where
  \begin{equation*}
    \Theta (x^i,y^i)
    \define
    \condset{(x^{*},y^{*})}{\forall x \in X. \ f(x^{*}, y^i) \leq
      f(x,y^i); \ \forall y \in Y. \ f(x^{*}, y^{*}) \leq f(x^{*}, y)}.
  \end{equation*}
  Suppose that Assumption~\ref{ass:general} holds and that the solution
  of the first optimization problem is always unique.
  Then every convergent subsequence of ${\left\{ (x^{i},y^i)
  \right\}}_{i=0}^{\infty}$ converges to a partial minimum.
  For two limit points $z, z'$ of such subsequences it holds that
  $f(z) = f(z')$.
\end{theorem}
Stronger results can be obtained if additional assumptions are made on
$\objFunADM$ and $\feasSet$:
If $\objFunADM$ is continuously differentiable, \refalgo{alg:std-adm} converges
to a stationary point (in the classical sense of \nonlinear
optimization).
If, in addition, $\objFunADM$ and $\feasSet$ are convex it is easy to show
that partial minimizers are also global minimizers of
Problem~\eqref{eq:abstract-problem}.
For more details on the convergence theory of classical ADMs,
\cf~\textcite{Gorski_et_al:2007} as well as
\textcite{Wendell_Hurter:1976}.

\subsection{Feasibility Pumps as ADMs}
\label{sec:fps-as-adms}

Recall that the basic feasibility pump
algorithm~\ref{alg:basic-fp-0-1-mips} tries to find a feasible solution
for the binary variant of \mip~\eqref{eq:MIP}.
We now consider the idealized feasibility pump, \ie we omit the
perturbation step in Line~\ref{alg:basic-fp-0-1-mips:perturb-step} of
\refalgo{alg:basic-fp-0-1-mips}, and show that the idealized
feasibility pump is a special case of the ADM
(Algorithm~\ref{alg:std-adm}) applied to a certain reformulation
of \mip~\eqref{eq:MIP}.
To this end, we duplicate the variables~$x_\intIdxSet$ using the new
variable vector~$y \in \set{0,1}^\intIdxSet$, yielding
\begin{align*}
  \min_{x,y} \quad
  & \transpose{c}x
  \\
  \st \quad
  & x \in X \define \condset{x \in \reals^n}{Ax \geq b, \
    x_\intIdxSet \in {[0,1]}^{\intIdxSet}},
  \\
  & y \in Y \define \set{0,1}^{\intIdxSet}, \quad \eqCons(x,y) = x_\intIdxSet -
   y = 0,
\end{align*}
which is obviously equivalent to the original \mip.
Note that, compared to the general
problem~\eqref{eq:abstract-problem}, we do not explicitly require the
inequality constraints vector~$\ineqCons$.
The feasibility pump is only interested in feasibility and thus
ignores the objective function.
By deleting the objective from the reformulated model and instead
moving an $\ell_1$ penalty term of the coupling condition~$y =
x_\intIdxSet$ into the objective, we obtain
\begin{subequations}
  \label{eq:adm-fp-problem}
  \begin{align}
    \min_{x,y} \quad
    & \norm{x_\intIdxSet - y}_1
    \\
    \st \quad
    & x \in X \define \condset{x \in \reals^n}{Ax \geq b, \
      x_\intIdxSet \in {[0,1]}^{\intIdxSet}},
    \quad
    y \in Y \define \set{0,1}^{\intIdxSet}.
  \end{align}
\end{subequations}
If we define initial values~$(x^0, y^0)$ by
$x^0 \define \argmin \condset{\transpose{c} x}{x \in X}$
and $y^0 \define \rounding{x^0}$,
it can be easily seen that solving Problem~\eqref{eq:adm-fp-problem}
with the ADM algorithm~\ref{alg:std-adm} exactly corresponds to
the idealized feasibility pump algorithm.
To be more precise, finding the new $x$-iterate within the ADM
coincides with the projection step and finding the new $y$-iterate
corresponds to the rounding step.

In the context of \refalgo{alg:std-adm}, we say that the sequence of
iterates~${z^k}$ cycles if there exists an iteration~$k$ and an $l
\geq 2$ with $z^k = z^{k+l}$.
Next, we prove that the ADM cannot cycle and thus terminates after a
finite number of iterations.
To this end, we make the following observations:
First, $X$ and $Y$ in~\eqref{eq:adm-fp-problem} are
non-empty and compact sets.
Second, we can assume uniqueness of the rounding step by
resolving tie-breaks choosing lexicographically minimal solutions.
Thus, by using that norms are continuous, we have the following
result.%
\begin{lemma}
  \refalgo{alg:std-adm} does not cycle.
\end{lemma}
\begin{proof}
  Assume the contrary, \ie there exists an iteration~$k$ and an $l \geq
  2$ such that $z^k, z^{k+1}, \dotsc, z^{k+l} = z^k$.
  Since $\objFunADM(x^{k+1}, y^{k+1}) \leq \objFunADM(x^{k+1}, y^k)
  \leq \objFunADM(x^k, y^k)$
  holds for all iterations~$k$ we directly see that $\objFunADM(z^k) =
  \objFunADM(z^{k+1}) = \dotsb = \objFunADM(z^{k+l-1})$ holds.
  This, however, implies that $z^k$ is already a partial minimum at
  which the algorithm stops.
\end{proof}
We note that this lemma is equivalent to Proposition~1 of
\textcite{DeSantis_et_al:2013}.
There, the authors show that the idealized feasibility pump for binary
\mips is equivalent to the Frank--Wolfe algorithm (using an unitary
stepsize) applied to the problem
\begin{equation}
  \label{eq:de-santis-problem}
  \min_{x \in \polyhedron} \quad \sum_{\intIdx \in \intIdxSet} \min
  \set{x_\intIdx, 1-x_\intIdx},
\end{equation}
where the objective function is a concave and \nonsmooth
merit function for measuring integrality.
Applying the convergence theory from~\cite{Mangasarian:1997} then
also yields finite termination at \socalled \emph{vertex
  stationary points}.
Since we have now proven that the ADM (Algorithm~\ref{alg:std-adm})
applied to~\eqref{eq:adm-fp-problem} is equivalent to the above mentioned
special case of the Frank--Wolfe method, we have also shown that partial
minima of Problem~\eqref{eq:adm-fp-problem} are exactly the vertex
stationary points of Problem~\eqref{eq:de-santis-problem}.

From the theory reported above and the last lemma we can directly
deduce the following convergence theorem for the idealized feasibility
pump.
\begin{theorem}
  The idealized feasibility pump terminates at a partial
  minimum~$(x^*, y^*)$ of Problem~\eqref{eq:adm-fp-problem} after a
  finite number of iterations.
  If the partial minimum~$(x^*, y^*)$ has objective
  value~$\norm{x^*_\intIdxSet - y^*}_1 = 0$, the
  point~$(x^*, y^*)$ is feasible for the \mip~\eqref{eq:MIP}.
\end{theorem}
This theorem also gives us a new view of the random perturbation steps
of feasibility pump algorithms:
They can be interpreted as an attempt to escape from non-integral
partial optima of Problem~\eqref{eq:adm-fp-problem}.

So far, we have only discussed the case of binary \mips.
However, our theory is still applicable as long as the optimization problems in
Line~\ref{alg:std-adm-problem-1} and~\ref{alg:std-adm-problem-2} of
\refalgo{alg:std-adm} are solved to global optimality.
This is a realistic assumption for convex \minlps of
type~\eqref{eq:MINLP}.
The suitable generalization of Problem~\eqref{eq:adm-fp-problem}
for this problem class reads
\begin{subequations}
  \label{eq:adm-fp-problem-minlp}
  \begin{align}
    \min_{x,y} \quad
    & \norm{x_\intIdxSet - y}_1
    \\
    \st \quad
    & x \in X \define \condset{x \in \reals^n}{\ineqCons(x) \geq 0,
      x_\intIdxSet \in {[0,1]}^\intIdxSet},
    \quad
    y \in Y \define \set{0,1}^\intIdxSet.
  \end{align}
\end{subequations}
We note that the resulting ADM is exactly the method presented in
\textcite{Bonami_Goncalves:2012}.
Using the same techniques as above, we get the following convergence
theorem for the idealized feasibility pump for convex \minlp.
\begin{theorem}
  The idealized feasibility pump for convex \minlp~\eqref{eq:MINLP} is
  equivalent to the ADM algorithm~\ref{alg:std-adm} applied to
  Problem~\eqref{eq:adm-fp-problem-minlp}.
  Thus, it terminates at a partial minimum~$(x^*,y^*)$ of
  Problem~\eqref{eq:adm-fp-problem-minlp} after a finite number of
  iterations.
  If this partial minimum has objective value~$\norm{x^*_\intIdxSet -
    y^*}_1 = 0$, the point~$(x^*, y^*)$ is feasible for the
  convex \minlp~\eqref{eq:MINLP}.
\end{theorem}

We close this section with two remarks.
First, we note that we presented the results in this section for
binary MI(NL)Ps only for improving readability.
The extension to general \mixedinteger problems is straightforward;
see~\refsec{sec:implementation-issues} for the details.
Second, we again want to highlight that the theoretical results
presented in this section only hold if the optimization problems in
Line~\ref{alg:std-adm-problem-1} and~\ref{alg:std-adm-problem-2} of
\refalgo{alg:std-adm} are solved to global optimality.
Since this is typically not possible for general \nonconvex \minlps,
the results are only practically valid for convex \mixedinteger
problems.


\section{The Penal\/ty Alternating Direction Method}
\label{sec:padm}

In this section we first present a new penalty alternating direction
method in~\refsec{sec:padm-algorithm} and afterward prove the convergence
results in~\refsec{sec:padm-conv-theory}.
Finally, in \refsec{sec:padm-as-fp} we show how the new method can be
used to obtain a novel feasibility pump algorithm for general \mixedinteger
optimization.
This new penalty alternating direction method based feasibility pump
replaces random perturbations with a theoretically analyzable penalty
framework for escaping from undesired intermediate points.
Thus, the complete theory presented for the new method also applies to
the new feasibility pump variant.

\subsection{The Algorithm}
\label{sec:padm-algorithm}

We now present the novel weighted $\ell_1$
penalty method based on the classical ADM framework given in
\refsec{sec:adm}.
To this end, we define the $\ell_1$ penalty function
\begin{equation*}
  \penFun_1(x,y; \penParEqu, \penParIneq)
  \define
  f(x,y) + \sum_{i=1}^\nrEqCons \penParEqu_i \abs{\eqCons_i(x,y)}
  + \sum_{i=1}^\nrIneqCons \penParIneq_i {[\ineqCons_i(x,y)]}^-,
\end{equation*}
where ${[\alpha]}^- \define \max \left\{ 0,-\alpha \right\}$ holds
and $\penParEqu = {(\penParEqu_i)}_{i=1}^m, \penParIneq =
{(\penParIneq_i)}_{i=1}^p \geq 0$ are the penalty parameters for the
equality and inequality constraints.
Note that we allow for different penalty parameters for the
constraints instead of a single penalty parameter as it is often the
case for penalty methods.

The penalty ADM now proceeds as follows.
Given a starting point and initial values for all penalty parameters,
the alternating direction method of \refalgo{alg:std-adm} is
used to compute a partial minimum of the penalty problem
\begin{equation}
  \label{eq:penalty-problem}
  \min_{x,y} \quad \penFun_1(x,y; \penParEqu, \penParIneq)
  \quad \st \quad x \in X, \ y \in Y.
\end{equation}
Afterward, the penalty parameters are updated and the next penalty
problem is solved to partial minimality.
Thus, the algorithm produces a sequence of partial minima of a
sequence of penalty problems of type~\eqref{eq:penalty-problem}.
More formally, the method is specified in \refalgo{alg:padm}.
\begin{algorithm}[tp]
  \caption{The $\ell_1$ Penalty  Alternating Direction Method}
  \label{alg:padm}
  \begin{algorithmic}[1]
    \STATE{Choose initial values $(x^{0,0}, y^{0,0}) \in X \times Y$
      and penalty parameters~$\penParEqu^0, \penParIneq^0 \geq 0$.}
    \FOR{$k=0,1,\dotsc$}%
    \STATE{Set $l = 0$.}
    \WHILE{$(x^{k,l}, y^{k,l})$ is not a partial minimum
      of~\eqref{eq:penalty-problem} with $\penParEqu = \penParEqu^k$
      and $\penParIneq = \penParIneq^k$}%
    \STATE{Compute $x^{k,l+1} \in \argmin_x \condset{\penFun_1(x,y^{k,l};
          \penParEqu^k, \penParIneq^k)}{x \in X}$.}
    \STATE{Compute $y^{k,l+1} \in \argmin_y \condset{\penFun_1(x^{k,l+1},y;
        \penParEqu^k, \penParIneq^k)}{y \in Y}$.}
    \STATE{Set $l \ot l+1$.}
    \ENDWHILE
    \STATE{Choose new penalty parameters~$\penParEqu^{k+1} \geq
      \penParEqu^{k}$ and $\penParIneq^{k+1} \geq
      \penParIneq^{k}$.}
    \ENDFOR
  \end{algorithmic}
\end{algorithm}


\subsection{Convergence Theory}
\label{sec:padm-conv-theory}

We now present the convergence results for
the penalty ADM algorithm~\ref{alg:padm}.
We start by proving that partial minima of the penalty problems are
partial minima of the original problem if they are feasible.
\begin{lemma}
  \label{lemma:feas-part-min-of-pen-problems}
  Assume that $(x^*, y^*)$ is a partial minimum of
  $\penFun_1(x,y;\penParEqu, \penParIneq)$ for arbitrary but fixed
  $\penParEqu, \penParIneq \geq 0$ and let
  $(x^*, y^*)$ be feasible for Problem~\eqref{eq:abstract-problem}.
  Then $(x^*, y^*)$ is a partial minimum of
  Problem~\eqref{eq:abstract-problem}.
\end{lemma}
\begin{proof}
  Let $x \in X$ such that $(x,y^*)$ is feasible for
  Problem~\eqref{eq:abstract-problem}.
  Then it holds
  \begin{align*}
    f(x,y^*)
    & =
    f(x,y^*) + \sum_{i=1}^\nrEqCons \penParEqu_i \abs{\eqCons_i(x,y^*)}
    + \sum_{i=1}^\nrIneqCons \penParIneq_i {[\ineqCons_i(x,y^*)]}^-
      \\
    & =
    \penFun_1(x,y^*;\penParEqu, \penParIneq)
      \geq
    \penFun_1(x^*,y^*;\penParEqu, \penParIneq)
    \\
    &=
    f(x^*,y^*) + \sum_{i=1}^\nrEqCons \penParEqu_i \abs{\eqCons_i(x^*,y^*)}
    + \sum_{i=1}^\nrIneqCons \penParIneq_i {[\ineqCons_i(x^*,y^*)]}^-
    \\
    &=
    f(x^*,y^*).
  \end{align*}
  The analogous inequality holds for all $y \in Y$ such that $(x^*,y)$
  is feasible.
  Thus, $(x^*,y^*)$ is a partial minimum of
  Problem~\eqref{eq:abstract-problem}.
\end{proof}

For the next theorem we need some more notation.
Let $\infeasMeas$ be the $\ell_1$ feasibility measure of
Problem~\eqref{eq:abstract-problem}, which we define as
\begin{equation*}
  \infeasMeas(x,y)
  \define
  \sum_{i=1}^\nrEqCons \abs{\eqCons_i(x,y)}
  + \sum_{i=1}^\nrIneqCons {[\ineqCons_i(x,y)]}^-.
\end{equation*}
Obviously, $\infeasMeas(x,y) \geq 0$ holds and
$\infeasMeas(x,y) = 0$ if and only if $(x,y)$ is feasible \wrt $g$ and
$h$.
Moreover, we define the weighted $\ell_1$ feasibility measure
as
\begin{equation*}
  \infeasMeas_{\penParEqu, \penParIneq}(x,y)
  \define
  \sum_{i=1}^\nrEqCons \penParEqu_i \abs{\eqCons_i(x,y)}
  + \sum_{i=1}^\nrIneqCons \penParIneq_i {[\ineqCons_i(x,y)]}^-,
\end{equation*}
\ie our $\ell_1$ penalty function can be stated as
\begin{equation*}
  \penFun_1(x,y; \penParEqu, \penParIneq)
  =
  f(x,y) + \infeasMeas_{\penParEqu, \penParIneq}(x,y).
\end{equation*}
The next theorem states that the sequence of partial minima of the
iteratively solved penalty problems converges to a partial minimum
of~$\infeasMeas_{\penParEqu, \penParIneq}$.
\begin{lemma}
  \label{lemma:conv-to-part-min-of-infeas-meas}
  Suppose that Assumption~\ref{ass:general} holds and that
  $\penParEqu^k_i \upto \infty$ for all $i = 1, \dotsc, m$
  and $\penParIneq^k_i \upto \infty$ for all $i = 1, \dotsc, p$.
  Moreover, let $(x^k,y^k)$ be a sequence of partial minima
  of~\eqref{eq:penalty-problem} (for $\penParEqu = \penParEqu^k$ and
  $\penParIneq = \penParIneq^k$) generated by \refalgo{alg:padm}
  with $(x^k,y^k) \to (x^*,y^*)$.
  Then there exist weights $\penParEquUb, \penParIneqUb \geq 0$
  such that $(x^*,y^*)$ is a partial minimizer of the feasibility measure
  $\infeasMeas_{\penParEquUb, \penParIneqUb}$.
\end{lemma}
\begin{proof}
  Let $(x^k,y^k)$ be a partial minimizer of $\penFun_1(x,y; \penParEqu^k,
  \penParIneq^k)$, \ie
  \begin{align*}
    \penFun_1(x,y^k; \penParEqu^k, \penParIneq^k)
    \geq
    \penFun_1(x^k,y^k; \penParEqu^k, \penParIneq^k)
    & \quad \text{for all } x \in X,
    \\
    \penFun_1(x^k,y; \penParEqu^k, \penParIneq^k)
    \geq
    \penFun_1(x^k,y^k; \penParEqu^k, \penParIneq^k)
    & \quad \text{for all } y \in Y,
  \end{align*}
  which is equivalent to
  \begin{equation}
    \label{eq:part-min-of-pen-func-1}
    \begin{split}
      &f(x,y^k) + \sum_{i=1}^\nrEqCons \penParEqu_i^k \abs{\eqCons_i(x,y^k)}
      + \sum_{i=1}^\nrIneqCons \penParIneq_i^k {[\ineqCons_i(x,y^k)]}^-
      \\
      \geq \ &
      f(x^k,y^k) + \sum_{i=1}^\nrEqCons \penParEqu_i^k \abs{\eqCons_i(x^k,y^k)}
      + \sum_{i=1}^\nrIneqCons \penParIneq_i^k {[\ineqCons_i(x^k,y^k)]}^-
    \end{split}
  \end{equation}
  for all $x \in X$ and
  \begin{equation}
    \label{eq:part-min-of-pen-func-2}
    \begin{split}
      & f(x^k,y) + \sum_{i=1}^\nrEqCons \penParEqu_i^k \abs{\eqCons_i(x^k,y)}
      + \sum_{i=1}^\nrIneqCons \penParIneq_i^k {[\ineqCons_i(x^k,y)]}^-
      \\
      \geq \ &
      f(x^k,y^k) + \sum_{i=1}^\nrEqCons \penParEqu_i^k \abs{\eqCons_i(x^k,y^k)}
      + \sum_{i=1}^\nrIneqCons \penParIneq_i^k {[\ineqCons_i(x^k,y^k)]}^-
    \end{split}
  \end{equation}
  for all $y \in Y$.
  The sequence~$(\penParEqu^k, \penParIneq^k) \subseteq \reals^{m+p}$ is
  unbounded but the normalized sequence
  \begin{equation*}
    \frac{(\penParEqu^k, \penParIneq^k)}{\norm{(\penParEqu^k,
        \penParIneq^k)}} \subseteq \reals^{m+p},
  \end{equation*}
  is bounded.
  Thus, there exists a subsequence (indexed by
  $l$) of the normalized sequence such that
  \begin{equation*}
    \frac{(\penParEqu^l, \penParIneq^l)}{\norm{(\penParEqu^l,
        \penParIneq^l)}}
    \to (\bar{\penParEqu}, \bar{\penParIneq})
    \quad \text{for } l \to \infty.
  \end{equation*}
  Division of~\eqref{eq:part-min-of-pen-func-1}
  and~\eqref{eq:part-min-of-pen-func-2} by
  $\norm{(\penParEqu^l, \penParIneq^l)}$ yields
  \begin{align*}
    &\frac{1}{\norm{(\penParEqu^l,
        \penParIneq^l)}} f(x,y^l)  + \sum_{i=1}^\nrEqCons
    \frac{\penParEqu_i^l}{\norm{(\penParEqu^l,
        \penParIneq^l)}} \abs{\eqCons_i(x,y^l)}
    + \sum_{i=1}^\nrIneqCons \frac{\penParIneq_i^l}{\norm{(\penParEqu^l,
        \penParIneq^l)}}
    {[\ineqCons_i(x,y^l)]}^-
    \\
    \geq \ &
    \frac{1}{\norm{(\penParEqu^l,
        \penParIneq^l)}}f(x^l,y^l) + \sum_{i=1}^\nrEqCons
    \frac{\penParEqu_i^l}{\norm{(\penParEqu^l,
        \penParIneq^l)}} \abs{\eqCons_i(x^l,y^l)}
    + \sum_{i=1}^\nrIneqCons \frac{\penParIneq_i^l}{\norm{(\penParEqu^l,
        \penParIneq^l)}} {[\ineqCons_i(x^l,y^l)]}^-
  \end{align*}
  for all $x \in X$ and
  \begin{align*}
    &\frac{1}{\norm{(\penParEqu^l,
        \penParIneq^l)}} f(x^l,y)  + \sum_{i=1}^\nrEqCons
    \frac{\penParEqu_i^l}{\norm{(\penParEqu^l,
        \penParIneq^l)}} \abs{\eqCons_i(x^l,y)}
    + \sum_{i=1}^\nrIneqCons \frac{\penParIneq_i^l}{\norm{(\penParEqu^l,
        \penParIneq^l)}}
      {[\ineqCons_i(x^l,y)]}^-
    \\
    \geq \ &
    \frac{1}{\norm{(\penParEqu^l,
        \penParIneq^l)}}f(x^l,y^l) + \sum_{i=1}^\nrEqCons
    \frac{\penParEqu_i^l}{\norm{(\penParEqu^l,
        \penParIneq^l)}} \abs{\eqCons_i(x^l,y^l)}
    + \sum_{i=1}^\nrIneqCons \frac{\penParIneq_i^l}{\norm{(\penParEqu^l,
             \penParIneq^l)}} {[\ineqCons_i(x^l,y^l)]}^-
  \end{align*}
  for all $y \in Y$.
  Finally, by using that the limit preserves non-strict inequalities,
  linearity of the limit, and continuity of $f,g$, and $h$, for $l
  \to \infty$ we obtain
  \begin{equation*}
    \sum_{i=1}^\nrEqCons
    \penParEquUb_i \abs{\eqCons_i(x,y^*)}
    + \sum_{i=1}^\nrIneqCons \penParIneqUb_i
    {[\ineqCons_i(x,y^*)]}^-
    \geq
    \sum_{i=1}^\nrEqCons
    \penParEquUb_i \abs{\eqCons_i(x^*,y^*)}
    + \sum_{i=1}^\nrIneqCons \penParIneqUb_i {[\ineqCons_i(x^*,y^*)]}^-
  \end{equation*}
  for all $x \in X$ and
  \begin{equation*}
    \sum_{i=1}^\nrEqCons
    \penParEquUb_i \abs{\eqCons_i(x^*,y)}
    + \sum_{i=1}^\nrIneqCons \penParIneqUb_i
    {[\ineqCons_i(x^*,y)]}^-
    \geq
    \sum_{i=1}^\nrEqCons
    \penParEquUb_i \abs{\eqCons_i(x^*,y^*)}
    + \sum_{i=1}^\nrIneqCons \penParIneqUb_i {[\ineqCons_i(x^*,y^*)]}^-
  \end{equation*}
  for all $y \in Y$.
  This is equivalent to
  \begin{align*}
    \infeasMeas_{\bar{\penParEqu}, \bar{\penParIneq}}(x,y^*)
    \geq
    \infeasMeas_{\bar{\penParEqu}, \bar{\penParIneq}}(x^*,y^*)
    & \quad \text{for all } x \in X,
    \\
    \infeasMeas_{\bar{\penParEqu}, \bar{\penParIneq}}(x^*,y)
    \geq
    \infeasMeas_{\bar{\penParEqu}, \bar{\penParIneq}}(x^*,y^*)
    & \quad \text{for all } y \in Y
  \end{align*}
  and thus completes the proof.
\end{proof}

The two preceding lemmas now enable us to characterize the
overall convergence behavior of the penalty \adm
algorithm~\ref{alg:padm}.

\begin{theorem}
  \label{thm:convergence}
  Suppose that Assumption~\ref{ass:general} holds and that
  $\penParEqu^k_i \upto \infty$ for all $i = 1, \dotsc, m$
  and $\penParIneq^k_i \upto \infty$ for all $i = 1, \dotsc, p$.
  Moreover, let $(x^k,y^k)$ be a sequence of partial minima
  of~\eqref{eq:penalty-problem} (for $\penParEqu = \penParEqu^k$ and
  $\penParIneq = \penParIneq^k$) generated by \refalgo{alg:padm}
  with $(x^k,y^k) \to (x^*,y^*)$.
  Then there exist weights $\penParEquUb, \penParIneqUb \geq 0$ such
  that $(x^*,y^*)$ is a partial minimizer of the feasibility measure
  $\infeasMeas_{\penParEquUb, \penParIneqUb}$.

  If, in addition, $(x^*,y^*)$ is feasible for the original
  problem~\eqref{eq:abstract-problem}, the following holds:
  \begin{enumerate}[label=\alph*)]
  \item If $f$ is continuous, then $(x^*,y^*)$ is a partial minimum of
    \eqref{eq:abstract-problem}.
  \item If $f$ is continuously differentiable, then $(x^*,y^*)$ is a
    stationary point of \eqref{eq:abstract-problem}.
  \item If $f$ is continuously differentiable and $f$ and $\Omega$ are
    convex, then $(x^*,y^*)$ is a global optimum of \eqref{eq:abstract-problem}.
  \end{enumerate}
\end{theorem}
\begin{proof}
  The first part always holds by
  Lemma~\ref{lemma:conv-to-part-min-of-infeas-meas}.
  If, in addition, the obtained partial minimum of
  $\infeasMeas_{\penParEquUb, \penParIneqUb}$ satisfies
  $\infeasMeas_{\penParEquUb, \penParIneqUb}(x^*,y^*) = 0$,
  we can apply Lemma~\ref{lemma:feas-part-min-of-pen-problems} and
  obtain the statements a)--c).
\end{proof}

In the next theorem we generalize the classical result on the
exactness of the $\ell_1$ penalty function (\cf, \eg
\cite{Han_Mangasarian:1979,Nocedal_Wright:2006}) to the setting of
partial minima.
For the ease of presentation, we state and prove this result only for the
case without inequality constraints.
However, the result can also be applied to problems including
inequality constraints by using standard reformulation techniques to
translate inequality constrained to equality constrained problems.
Beforehand, we need two assumptions:
\begin{assumption}
  \label{ass:loc-part-lipschitz}
  The objective function $f : X \times Y \to \reals$ of
  Problem~\eqref{eq:abstract-problem} is locally
  Lipschitz continuous in the direction of $x$ and of $y$, \ie for
  every $(x^*,y^*) \in \feasSet$ there exists an open set~$N(x^*,y^*)$
  containing $(x^*,y^*)$ and a constant~$L \geq 0$ such that
  \begin{align*}
    \abs{f(x, y^*) - f(x^*, y^*)}
    \leq L \norm{x - x^*}
    &\quad \text{for all } x \text{ with } (x,y^*) \in N(x^*,y^*),
    \\
    \abs{f(x^*, y) - f(x^*, y^*)}
    \leq L \norm{y - y^*}
    &\quad \text{for all } y \text{ with } (x^*,y) \in N(x^*,y^*).
  \end{align*}
\end{assumption}

Note that if one set, say $Y$, is discrete, the corresponding
condition is trivially satisfied.
In this case any set of the form $(U, \set{y^{*}})$, where $U
\subseteq X$ is an open neighborhood around $x^{*}$, is an open
neighborhood around $(x^{*}, y^{*})$.

\begin{assumption}
  \label{ass:g-direct-deriv-cond}
  For every constraint~$\eqCons_i, i=1,\dotsc,m$, there exists a
  constant~\mbox{$l_i > 0$} such that
  \begin{align*}
    l_i \norm{x - x^*} \leq \abs{\eqCons_i(x, y^*) - \eqCons_i(x^*,y^*)}
    & \quad \text{for all } x \text{ with } (x,y^*) \in N(x^*,y^*),
    \\
    l_i \norm{y - y^*} \leq \abs{\eqCons_i(x^*, y) - \eqCons_i(x^*,y^*)}
    & \quad \text{for all } y \text{ with } (x^*,y) \in N(x^*,y^*).
  \end{align*}
\end{assumption}
Note that in the case of existing directional derivatives of $\eqCons_i$, the
latter assumption states that the directional derivatives of the $\eqCons_i$
both in the direction of $x$ and of $y$ are bounded away from zero.
Before we state and proof the exactness theorem we brief\/ly discuss
the latter assumption.
In the context of \adms, the constraints~$\eqCons(x,y)=0$ are mostly
so-called copy constraints of the type
\begin{equation*}
  \eqCons(x,y) = A(x-y) = 0
\end{equation*}
that are used to decompose the genuine problem formulation such that
it fits into the framework of Problem~\eqref{eq:abstract-problem};
\cf, \eg \textcite{Nowak:2005}.
If the matrix~$A$ is square and has full rank---as it
is typically the case for copy constraints---the constraints~$g$ are
bi-Lipschitz and thus fulfill Assumption~\ref{ass:g-direct-deriv-cond}.
Now, we are ready to state and prove the theorem on exactness of
the $\ell_1$ penalty function \wrt partial minima.
\begin{theorem}
  Let $(x^*,y^*)$ be a partial minimizer of
  \begin{equation}
    \label{eq:abstract-problem-wo-ineqs}
    \min_{x,y} \quad f(x,y)
    \quad \st \quad
    \eqCons(x,y) = 0, \ x \in X, \ y \in Y,
  \end{equation}
  and suppose that Assumptions~\ref{ass:loc-part-lipschitz}
  and~\ref{ass:g-direct-deriv-cond} hold.
  Then there exists a constant~$\penParEquUb > 0$ such that
  $(x^*,y^*)$ is a partial minimizer of
  \begin{equation*}
    \min_{x,y} \quad \penFun_1(x,y; \penParEqu)
    \quad \st \quad x \in X, \ y \in Y
  \end{equation*}
  for all $\penParEqu \geq \penParEquUb$
  and
  \begin{equation*}
    \penFun_1(x,y; \penParEqu)
    \define
    f(x,y) + \sum_{i=1}^\nrEqCons \penParEqu_i \abs{\eqCons_i(x,y)}.
  \end{equation*}
\end{theorem}
\begin{proof}
  Since $(x^*,y^*)$ is a partial minimizer of
  Problem~\eqref{eq:abstract-problem-wo-ineqs}, it holds
  that
  \begin{subequations}
    \label{eq:z-star-part-min-nlp}
    \begin{align}
      f(x,y^*) \geq f(x^*, y^*)
      & \quad \text{for all } (x,y^*) \in \feasSet,
      \\
      f(x^*,y) \geq f(x^*, y^*)
      & \quad \text{for all } (x^*,y) \in \feasSet,
    \end{align}
  \end{subequations}
  where $\feasSet$ is the feasible region of
  Problem~\eqref{eq:abstract-problem-wo-ineqs}.
  First, assume that $(x,y^*)$ is feasible for
  Problem~\eqref{eq:abstract-problem-wo-ineqs}.
  Using~\eqref{eq:z-star-part-min-nlp} we obtain
  \begin{align*}
    \penFun_1(x^*,y^*; \penParEqu)
    & = f(x^*,y^*) + \sum_{i=1}^\nrEqCons \penParEqu_i \abs{\eqCons_i(x^*,y^*)}
    \\
    & = f(x^*,y^*) \leq f(x,y^*)
    \\
    & = f(x,y^*) + \sum_{i=1}^\nrEqCons \penParEqu_i \abs{\eqCons_i(x,y^*)}
      \\
    & = \penFun_1(x,y^*; \penParEqu)
  \end{align*}
  for all $\penParEqu$.
  The inequality
  \begin{equation*}
    \penFun_1(x^*,y^*; \penParEqu)
    \leq
    \penFun_1(x^*,y; \penParEqu)
  \end{equation*}
  can be shown analogously assuming that $(x^{*},y)$ is feasible.

  We now consider the case that $(x,y^*)$ is not feasible for
  Problem~\eqref{eq:abstract-problem-wo-ineqs}.
  We set $\penParEquUb \define L / (\nrEqCons \bar{l}) e > 0$,
  where $\bar{l} \define \min_{i=1,\dots,\nrEqCons}\set{l_i}, e =
  {(1,\dotsc,1)}^\top \in \reals^\nrEqCons$, and show that
  for all $\penParEqu \geq \bar{\penParEqu}$ the inequality
  \begin{equation}
    \label{eq:p-min-pen-func-def-x}
    f(x,y^*) + \sum_{i=1}^\nrEqCons \penParEqu_i \abs{\eqCons_i(x,y^*)}
    \geq
    f(x^*,y^*) + \sum_{i=1}^\nrEqCons \penParEqu_i \abs{\eqCons_i(x^*,y^*)}
  \end{equation}
  holds for all $x \in X$.
  Since $(x^*,y^*)$ is feasible for
  Problem~\eqref{eq:abstract-problem-wo-ineqs},
  Inequality~\eqref{eq:p-min-pen-func-def-x} is equivalent to
  \begin{equation*}
    \sum_{i=1}^\nrEqCons \penParEqu_i \abs{\eqCons_i(x,y^*)}
    \geq
    f(x^*,y^*) - f(x,y^*).
  \end{equation*}
  In the following, we use that for each $(x, y) \in X \times Y$,
  Assumption~\ref{ass:loc-part-lipschitz} also implies the
  global Lipschitz continuity of $f$ on the compact sets $X \times
  \set{y}$ and $\set{x} \times Y$.
  From the definition of $\penParEquUb$ and the
  Assumptions~\ref{ass:loc-part-lipschitz} and
  \ref{ass:g-direct-deriv-cond} we obtain
  \begin{align*}
    & \sum_{i=1}^\nrEqCons \penParEqu_i \abs{\eqCons_i(x,y^*)}
      \geq \penParEquUb \sum_{i=1}^\nrEqCons \abs{\eqCons_i(x,y^*)}
      = \penParEquUb \sum_{i=1}^\nrEqCons \abs{\eqCons_i(x,y^*) - \eqCons_i(x^*,y^*)}
    \\
    \geq \ & \penParEquUb \sum_{i=1}^\nrEqCons l_i \norm{x - x^*}
             \geq \penParEquUb \nrEqCons \bar{l} \norm{x - x^*}
             = L \norm{x - x^*}
    \\
    \geq \ & \abs{f(x^*,y^*) - f(x,y^*)}
             \geq f(x^*,y^*) - f(x,y^*).
  \end{align*}
  Thus,~\eqref{eq:p-min-pen-func-def-x} holds for all $\penParEqu \geq
  \penParEquUb$.
  Analogously, it can be shown that the inequality
  \begin{equation*}
    f(x^*,y) + \sum_{i=1}^\nrEqCons \penParEqu_i \abs{\eqCons_i(x^*,y)}
    \geq
    f(x^*,y^*) + \sum_{i=1}^\nrEqCons \penParEqu_i \abs{\eqCons_i(x^*,y^*)}
  \end{equation*}
  holds for all $y \in Y$.
\end{proof}


\subsection{The Penalty ADM as a Feasibility Pump}
\label{sec:padm-as-fp}

In this section we discuss the application of the proposed penalty
alternating direction method as a new variant of a feasibility pump
algorithm for convex \minlps.
The motivation is the following:
We have seen in the last section that idealized, \ie
perturbation-free, feasibility pumps for convex \minlps terminate at 
partial minima after a finite number of iterations.
However, it is possible that the obtained partial minimum is not
integer feasible.
Feasibility pumps typically try to resolve this problem by applying a
random perturbation of the integer components.
This procedure has the significant drawback that it renders a
convergence theory of the overall method (almost) impossible.
In contrast to these random perturbations, the method we propose uses
a theoretically analyzable penalty framework to escape integer infeasible
partial minima.

We start by rewriting Problem~\eqref{eq:MINLP} by again duplicating
the integer components~$x_\intIdxSet$ of $x$ and obtain
\begin{equation}
  \label{eq:MINLP-reformulated}
  \min_{x,y} \quad
  \objFun(x) \quad
  \st \quad
  \ineqCons(x) \geq 0,
  \quad
  x_\intIdxSet = y,
  \quad
  y \in \integers^\intIdxSet \cap [l_\intIdxSet, u_\intIdxSet] .
\end{equation}
With the compact sets
\begin{align*}
  X \define \condset{x}{\ineqCons(x) \geq 0},
  \quad
  Y \define \integers^\intIdxSet \cap [l_\intIdxSet, u_\intIdxSet]
\end{align*}%
and the additional equality constraints
\begin{align*}
  \eqCons(x,y) = x_\intIdxSet - y = 0
\end{align*}
we can apply \refalgo{alg:padm} to
Problem~\eqref{eq:MINLP-reformulated}.
Note that the problem interfaced to the penalty ADM does not contain
any inequality constraints explicitly since we moved them to the
set~$X$.
This also simplifies the $\ell_1$ penalty function $\penFun_1(x,y;
\penParEqu, \penParIneq)$ to $\penFun_1(x,y; \penParEqu$).
In the $l$th ADM iteration of the $k$th penalty iteration of 
\refalgo{alg:padm} the two subproblems being solved 
are
\begin{equation*}
  \min_{x \in X} \quad \penFun_1(x,y^{k,l}; \penParEqu^k),
\end{equation*}
which can be written as
\begin{equation}
  \label{eq:NLP}
  \min_x \quad
  f(x) + \sum_{\intIdx \in \intIdxSet} \penParEqu_i^k \abs{x_i - y_i^{k,l}} 
  \quad \st \quad \ineqCons(x) \geq 0,
\end{equation}
and
\begin{equation*}
  \min_{y \in Y} \quad \penFun_1(x^{k,l+1},y;  \penParEqu^k),
\end{equation*}
which can be written as
\begin{equation}
  \label{eq:RND}
  \min_y \quad f(x^{k,l+1}) + \sum_{\intIdx \in \intIdxSet}
  \penParEqu^{k}_i \abs{x_i^{k,l+1} - y_i}
  \quad \st \quad 
  y \in \integers^\intIdxSet \cap [\lb_\intIdxSet, \ub_\intIdxSet].
\end{equation}
Note that Problem~\eqref{eq:NLP} is the \nlp relaxation
of~\eqref{eq:MINLP}, where the original objective function is augmented
by a weighted $\ell_1$ penalty term.
Note further that solving Problem~\eqref{eq:RND} simply
means to apply a weighted rounding of the variables~$y_i, \intIdx \in
\intIdxSet$.



\section{Implementation Issues}
\label{sec:implementation-issues}

In this section we comment on important implementation issues.
First, we rewrite Problem~\eqref{eq:MINLP-reformulated} by
replacing the coupling equality~$x_\intIdxSet = y$ by inequality
constraints in order to be able to penalize a coupling
error~$x_\intIdx > y_\intIdx, \intIdx \in \intIdxSet$, different
than the error~$y_\intIdx > x_\intIdx$.
In addition, we also explicitly state all variable bounds from now
on.
Thus, we obtain
\begin{align*}
  \min_{x,y} \quad & \objFun(x)\\
  \st \quad & \ineqCons(x) \geq 0, \quad x \in [l,u],\\
                   & x_\intIdxSet \geq y, \quad
                     y \geq x_\intIdxSet,
                     \quad y \in \integers^\intIdxSet \cap
                     [l_\intIdxSet, u_\intIdxSet].
\end{align*}
In other words, we slightly modified the compact and non-empty
constraint sets to
\begin{align*}
  X \define \condset{x \in [l,u]}{\ineqCons(x) \geq 0},
  \quad
  Y \define \integers^\intIdxSet \cap [l_\intIdxSet, u_\intIdxSet]
\end{align*}
and replaced the coupling equalities by coupling inequalities.
The two subproblems~\eqref{eq:NLP} and~\eqref{eq:RND} that are solved
within the $l$th ADM iteration of the $k$th penalty iteration
are now given by
\begin{subequations}
  \label{eq:NLP-rewritten}
  \begin{align}
    \min_x \quad
    & f(x) + \sum_{\intIdx \in \intIdxSet}
      \left(\ubar{\penParIneq}^{k}_i {[x_i - y_i^{k,l}]}^- +
      \bar{\penParIneq}^{k}_i {[y_i^{k,l} - x_i]}^-\right)
    \\
    \st \quad & \ineqCons(x) \geq 0, \quad x \in [l,u],
  \end{align}
\end{subequations}
and
\begin{subequations}
  \label{eq:RND-rewritten}
  \begin{align}
    \min_y \quad
    & f(x^{k,l+1}) + \sum_{\intIdx \in \intIdxSet}
      \left(\ubar{\penParIneq}^{k}_i {[x_i^{k,l+1} - y_i]}^- +
      \bar{\penParIneq}^{k}_i {[y_i - x_i^{k,l+1}]}^-\right)
    \\
    \st \quad & y \in \integers^\intIdxSet \cap [l_\intIdxSet, u_\intIdxSet],
  \end{align}
\end{subequations}
\ie we also replaced the single penalty parameters~$\penParEqu_i$ for
the equality coupling constraints in
Problem~\eqref{eq:MINLP-reformulated} by two new penalty
parameters~$\ubar{\penParIneq}_i$ and $\bar{\penParIneq}_i$ for the
lower and upper violation of the coupling.

In order to actually implement the penalty ADM based feasibility pump
for \minlps, we follow \textcite{Achterberg_Berthold:2007} and scale
the objective function of Problem~\eqref{eq:NLP-rewritten} such that
the impact of the $\ell_1$~penalty terms and the original objective
function~$\objFun$ can be balanced.
This balancing between feasibility and optimality is
done using the parameter~$\alpha^k \in [0,1]$.
Additionally, we again rewrite the $\ell_1$~penalty terms in the
objective function.
To this end, we denote the set of indices of binary
variables by $\binIdxSet \subseteq \intIdxSet$,
introduce the variables~$d_\intIdx^+, d_\intIdx^- \geq 0$
for all $\intIdx \in \intIdxSet \setminus \binIdxSet$, and rewrite
Problem~\eqref{eq:NLP-rewritten} as
\begin{equation}
  \label{eq:NLP-dash}
  \begin{split}
    \min_{x,d} \quad & \alpha^k
    \frac{\sqrt{|\intIdxSet|}}{\norm{\nabla \objFun(x^{(0,0)})}}\, \objFun(x) +
    (1-\alpha^k) \tilde{\infeasMeas}(x_\binIdxSet, d_{\intIdxSet
      \setminus \binIdxSet}; \penParIneq^\pm_\intIdxSet)\\
    \st \quad & \ineqCons(x) \geq 0, \quad x \in [l,u],\\
    & d_\intIdx^+ \geq x_i - y_i^{k,l} \qConstraintFor{\intIdx \in
      \intIdxSet \setminus \binIdxSet},\\
    & d_\intIdx^- \geq y_i^{k,l} - x_i \qConstraintFor{\intIdx \in
      \intIdxSet \setminus \binIdxSet},\\
    & d_\intIdx^-, d_\intIdx^+ \geq 0
  \end{split}
\end{equation}
with
\begin{equation*}
  \tilde{\infeasMeas}(x_\binIdxSet, d_{\intIdxSet
    \setminus \binIdxSet}; \penParIneq^\pm_\intIdxSet)
  \define
  \sum_{\intIdx \in \binIdxSet_0} \ubar{\penParIneq}^{k}_i x_i
  + \sum_{\intIdx \in \binIdxSet_1} \bar{\penParIneq}^{k}_i (1 - x_i)
  + \sum_{\intIdx \in \intIdxSet \setminus \binIdxSet}
  \left(\ubar{\penParIneq}_\intIdx^{k} d_\intIdx^+
    + \bar{\penParIneq}_\intIdx^{k} d_\intIdx^-\right),
\end{equation*}
where $B_0 \define \condset{i \in \binIdxSet}{y_i^{k,l}=0}$
and $B_1 \define \condset{i \in \binIdxSet}{y_i^{k,l}=1}$.
For binary \minlps, \ie $\intIdxSet = \binIdxSet$, only the objective
function of Problem~\eqref{eq:NLP-dash} may change from one iteration to
the next.
This is of special importance for \mixedinteger linear problems since
then the optimal simplex basis obtained in iteration~$k$ yields
a primal feasible starting basis for iteration~$k+1$.
However, when $\intIdxSet \neq \binIdxSet$,
the optimal basis obtained from iteration~$k$
is generally (primal and dual) infeasible for the \lp that has to be
solved in iteration~$k+1$.

Since $\objFun(x^{k,l+1})$ is constant and $\alpha^k \in [0,1]$, solving
Problem~\eqref{eq:RND-rewritten} is equivalent to solving the
$\card{\intIdxSet}$~independent problems
\begin{equation*}
  y_\intIdx^{k,l+1} \define \argmin_{y_\intIdx}
  \condSet{\ubar{\penParIneq}^{k}_\intIdx {[x_\intIdx^{k,l+1} - y_\intIdx]}^-
    + \bar{\penParIneq}^{k}_\intIdx {[y_\intIdx - x_\intIdx^{k,l+1}]}^-}{y_\intIdx \in
    \integers \cap [l_\intIdx, u_\intIdx]},
  \quad \intIdx \in \intIdxSet.
\end{equation*}
The solutions to these problems can be stated explicitly:
\begin{equation*}
  y_\intIdx^{k,l+1} =
  \begin{cases}
    \lceil x^{k,l+1}_\intIdx \rceil,
    & \text{if } \bar{\penParIneq}^{k}_\intIdx (\lceil x^{k,l+1}_\intIdx \rceil - x^{k,l+1}_\intIdx) \leq
    \ubar{\penParIneq}^{k}_\intIdx (x^{k,l+1}_\intIdx - \lfloor x^{k,l+1}_\intIdx \rfloor),
    \\
    \lfloor x^{k,l+1}_\intIdx \rfloor, & \text{otherwise}.
  \end{cases}
\end{equation*}

Finally, we have a look on the update of the penalty parameters. An
update takes place, whenever the inner ADM loop terminates. In our
implementation, we terminate the $k$th penalty iteration if $\|(x^{k,l},
y^{k,l}) - (x^{k,l-1}, y^{k,l-1})\|_{\infty} \leq \epsilon$ holds,
where $\epsilon = 10^{-5}$.
For the actual update of the penalty parameters, we set
\begin{align*}
  \ubar{\penParIneq}^{k+1}_\intIdx
  & =
    \begin{cases}
      \inc(\ubar{\penParIneq}^{k}_\intIdx), & \text{if }  y_\intIdx^{k,l+1} = \lceil x^{k,l+1}_\intIdx \rceil,\\
      \ubar{\penParIneq}^{k}_\intIdx, & \text{otherwise},
    \end{cases}
  \\
  \bar{\penParIneq}^{k+1}_\intIdx
  &=
    \begin{cases}
      \inc(\bar{\penParIneq}^{k}_\intIdx), & \text{if }  y_\intIdx^{k,l+1} = \lfloor x^{k,l+1}_\intIdx \rfloor,\\
      \bar{\penParIneq}^{k}_\intIdx, & \text{otherwise},
    \end{cases}
\end{align*}
where the penalty parameter update operator~$\inc(a)$ may be any function
with \mbox{$\inc(a) > a$}, \eg $\inc(a) = a+1$ or $\inc(a) = 10a$ are
used in our computational study.
This way, unsuccessful rounding down of the same
variable is eventually followed by rounding up of this variable due
to increasingly penalizing rounding down and vice versa.
Similarly, the conventional feasibility pump algorithm tries to escape
from repeated rounding in the \enquote{wrong} direction by randomly
switching the rounding direction from time to time;
\cf~\textcite{Fischetti_et_al:2005}.
Finally, we set $\alpha^{k+1} = \lambda \alpha^k$ with $\lambda \in
(0,1)$ whenever the penalty parameters are updated.

We note that choosing $\alpha_0 = 0$ in our penalty alternating
direction method based feasibility pump is similar to the
feasibility pump presented in
\textcite{Fischetti_et_al:2005,Bertacco_et_al:2007}, while choosing
$\alpha_0 = 1$ yields an algorithm the behaves similar to the
objective feasibility pump algorithm presented by
\textcite{Achterberg_Berthold:2007}.
Lastly, we note that for $\alpha_0 = 0$, we also have $\alpha_k = 0$ for
all $k > 0$.
In this case the first term of the objective function of
Problem~\eqref{eq:NLP-dash} vanishes for all $k,l$.
In any other case, we can divide the entire objective function by
$\alpha_k$ to be conformal to the theoretical setting presented in
previous sections.


\section{Computational Results}
\label{sec:comp-results}

In this section we present extensive numerical results for the penalty
ADM based feasibility pump introduced in \refsec{sec:padm} and
\ref{sec:implementation-issues}.
Since our method is completely generic in terms of the problem type to
which it is applied, we present computational results
both for \mips in \refsec{sec:comp-results:mip}
and for (convex as well as \nonconvex) \minlps in
\refsec{sec:comp-results:minlp}.

Throughout this section we use log-scaled performance profiles as
proposed
\ifPreprint
by~\textcite{Dolan_More:2001}
\fi
to compare running times and
solution quality.
As it is always the case for log-scaled performance profiles the axes
have the following meaning:
If $(x,y)$ lies on a profile curve, this means that the respective
solver is not more than $2^x$-times worse than the best solver on
$100y$\,\si{\percent} of the problems of the test set.
Running times are always given in seconds and,
following~\cite{Koch_et_al:2010}, solution quality is measured by the
primal-dual gap defined by
\begin{equation}\label{eq:primal-dual-gap-def}
  \text{gap} = \frac{b_{\text{p}} - b_{\text{d}}}{\inf
    \condset{\abs{z}}{z \in [b_{\text{d}}, b_{\text{p}}]}},
\end{equation}
where $b_{\text{p}}$ is the primal and $b_{\text{d}}$ is the dual
bound, respectively.
Additionally, we set
$\text{gap}=\infty$ whenever $b_{\text{d}} < 0 \leq b_{\text{p}}$
and $\text{gap}=0$ if $b_{\text{d}} = b_{\text{p}} = 0$.

All computational experiments have been executed on a 12~core
Xeon~5650 ``Westmere'' chip running at \SI{2.66}{\giga\hertz} with
\SI{12}{\mega\byte}~shared cache per chip and \SI{24}{\giga\byte} of
DDR3-1333~RAM\@.
The time limit is set to $t^+ = \SI{1}{\hour}$ without any limit on the
number of iterations for the outer penalty and the inner \adm loop.
Additional information about the computational setup and the
implementation details are given in the respective sections.

\subsection{Mixed-Integer Linear Problems}
\label{sec:comp-results:mip}

We start with discussing the results of our algorithm
applied to \mixedinteger linear problems.
For \mips, our algorithm is implemented in \Cpp and uses
\Gurobi~6.5.0~\cite{Gurobi650} for solving the \lp subproblems.
We use \Gurobi's option \texttt{deterministic} \texttt{concurrent}
for the first \lp and solve all succeeding \lps using the primal simplex
method; \cf~\refsec{sec:implementation-issues}.
The \Cpp code has been compiled with \gcc~4.8.4 using the optimization
flag~\texttt{o3}.
\begin{figure}[tp]
  \centering
  \includegraphics[width=\textwidth]{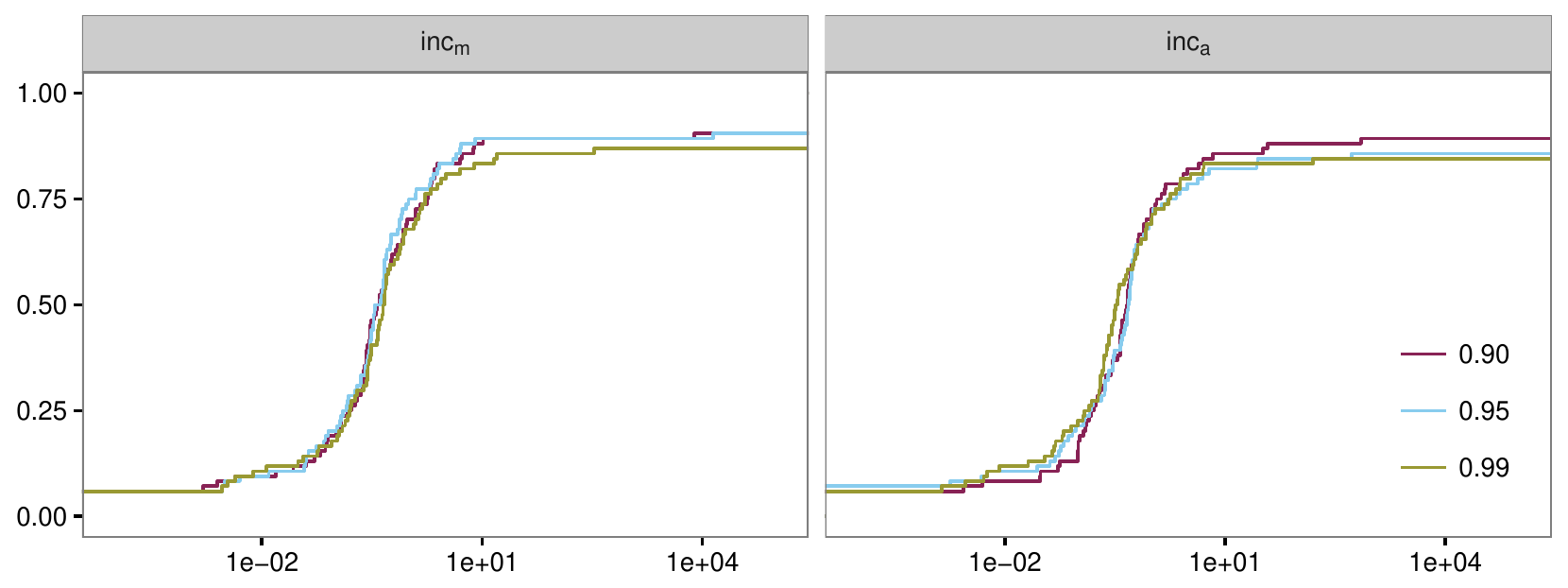}%
  \\
  \includegraphics[width=\textwidth]{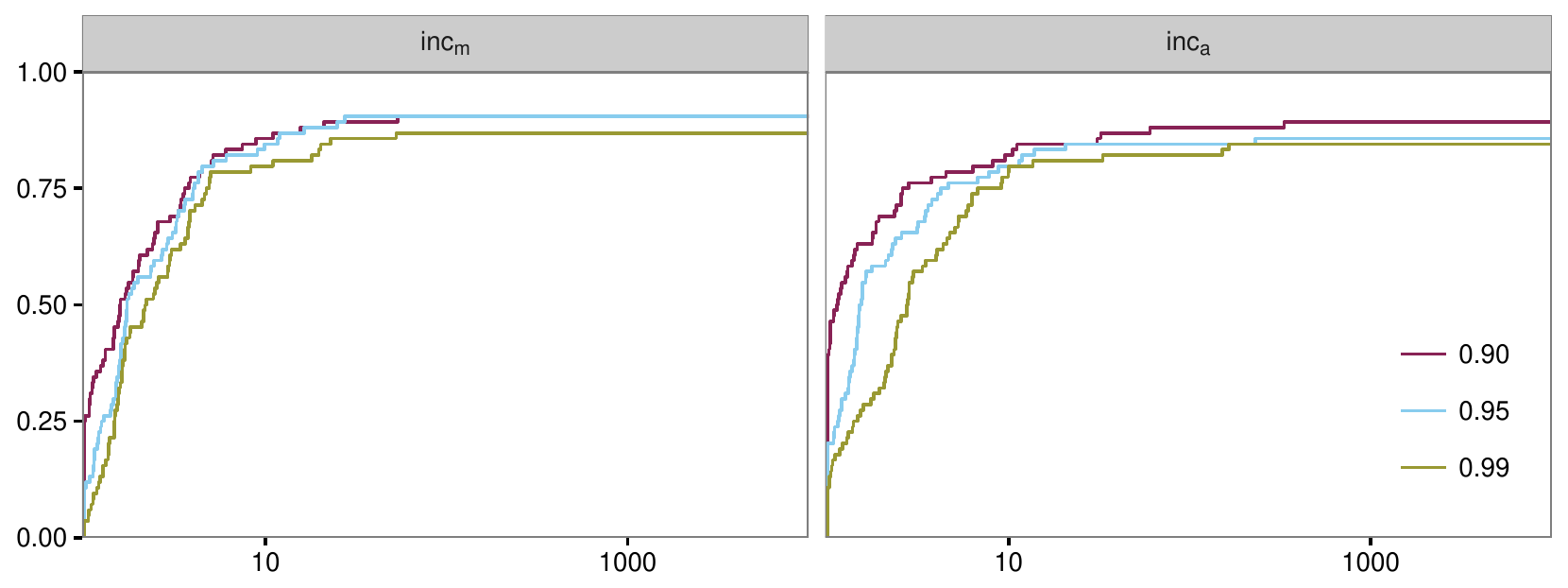}%
  \caption{Non-dominated parameterizations $\lambda \in \set{0.9,
      0.95, 0.99}$, $\inc \in \set{\inc_{\mathrm{m}},
      \inc_{\mathrm{a}}}$ (all with activated \Gurobi presolve) for
    the penalty ADM based feasibility pump for \mips. Top: primal-dual
    gap. Bottom: running times.}
  \label{fig:miplib02010-nondominated}
\end{figure}%
First, we present a parameter study.
Our penalty \adm based feasibility pump can be instantiated using
different choices for certain algorithmic parameters:
The initial convex combination parameter~$\alpha^0$ for
weighting the objective function and the distance
function~$\tilde{\infeasMeas}$ (\cf~Problem~\eqref{eq:NLP-dash}) is
always set to $\alpha^0 = 1$, emulating the objective feasibility pump
of \textcite{Achterberg_Berthold:2007}.
The parameter~$\lambda$ for updating the convex combination
parameter is varied in the set~$\set{0.9, 0.95, 0.99}$ in our
parameter study and the penalty parameter update operator can
be chosen to be the additive variant~$\inc_{\mathrm{a}}(x) = x+1$ or the
multiplicative variant~$\inc_{\mathrm{m}}(x) = 10x$.
In addition, we also tested the impact of (de)activating the \mip presolve
of \Gurobi before applying our method.
Thus, combining the different choices for $\lambda, \inc$, and the
(de)activation of \Gurobi's presolve leads to 12~parameter
combinations.
We applied every of these 12~variants of our method to solve the
\MIPLIBtwentyten benchmark test set excluding the infeasible instances
\codeName{ash608gpia-3col}, \codeName{enlight14},
and \codeName{ns1766074}.
This yields a test set of 84~instances;
\cf~\textcite{Koch_et_al:2010}.
In order to determine the best parameterization, we compare all
12~variants using performance profiles, where the performance
measure is chosen as defined in~\eqref{eq:primal-dual-gap-def}.
We then exclude a parameterization~$p$ if another
parameterization~$p'$ exists that dominates $p$.
Here, domination is defined by a performance profile completely
left-above the other one.
This yields the exclusion of deactivating \Gurobi's presolve
and, thus, 6~remaining parameterizations;
$\lambda \in \set{0.9,0.95,0.99}$ and
$\inc \in \set{\inc_{\mathrm{m}}, \inc_{\mathrm{a}}}$.
The corresponding performance profiles are given in
\reffig{fig:miplib02010-nondominated}.
\begin{figure}[tp]
  \centering
  \includegraphics[width=0.48\textwidth]{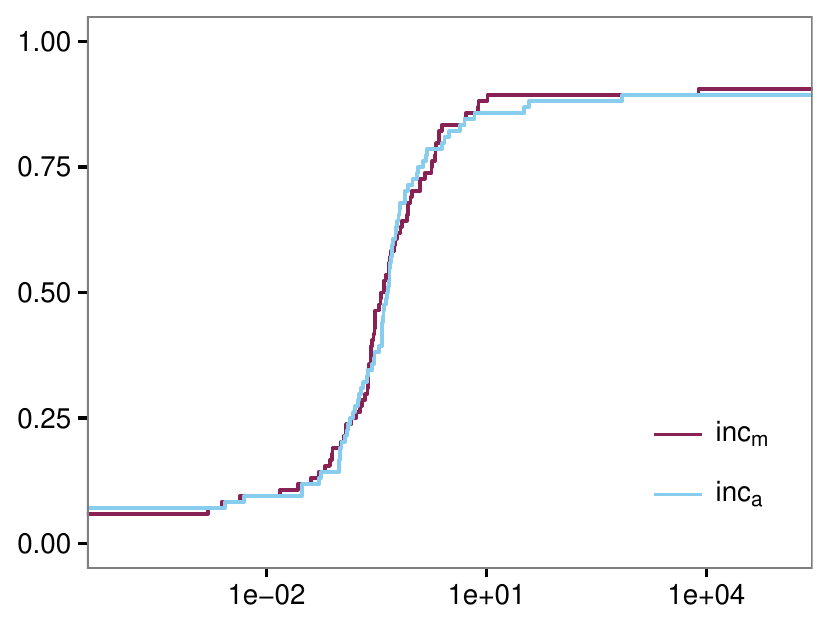}%
  \quad
  \includegraphics[width=0.48\textwidth]{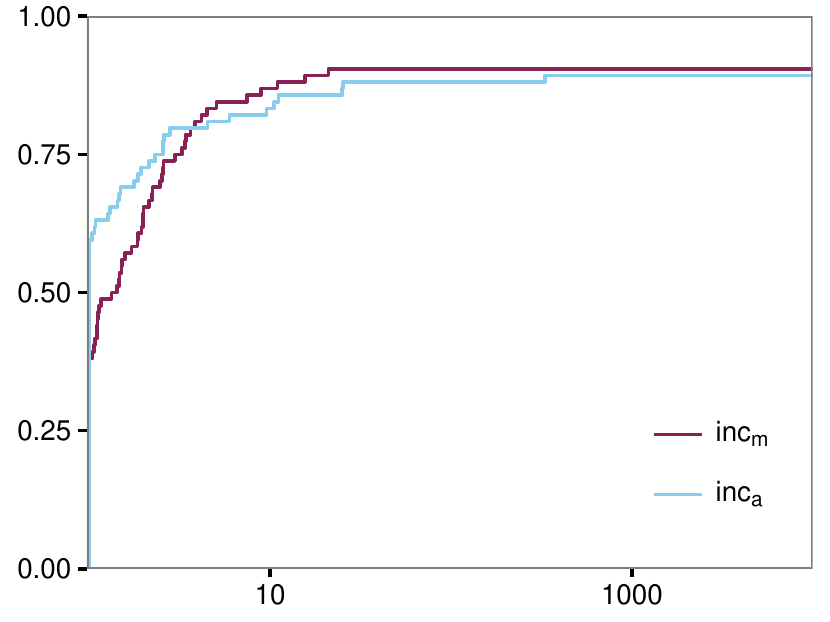}%
  \caption{Performance profiles of primal-dual gap (left) and running
    times (right) for the winner parameterizations $\lambda = 0.9$,
    activated \Gurobi presolve, and
    $\inc \in \set{\inc_{\mathrm{m}}, \inc_{\mathrm{a}}}$ for \mips}
  \label{fig:miplib2010-winners}
\end{figure}%
It can be seen that lower values for $\lambda$
yield more robust instantiations of the algorithm,
\ie the number of instances for which a feasible solution can be found
is larger.
Additionally, all tested variants solve 5 out of 84~instances to global
optimality, except for the variant with $\lambda = 0.95$ and
$\inc_{\text{a}}$ penalty parameter update rule, which solves
6~instances to global optimality.
Altogether, the six parameter choices are quite comparable.
Turning to running times, it can be clearly seen that smaller values
of $\lambda$ also lead to shorter running times.
Thus, our parameter study suggests to activate the \mip presolve of
\Gurobi, to choose $\lambda = \num{0.9}$, and to leave the choice of
the penalty parameter update rule~$\inc \in
\set{\inc_{\mathrm{m}}, \inc_{\mathrm{a}}}$ as an option for the
user.
\reffig{fig:miplib2010-winners} shows the performance profiles
for solution quality (left) and running times (right) for these
``winning'' parameterizations.
We again see that some instances are solved to global
optimality\footnote{The instances
  \codeName{triptim1}, \codeName{pigeon-10}, \codeName{enlight13},
  \codeName{ex9}, and \codeName{ns1758913} are solved to
  global optimality using the $\inc_{\text{m}}$ penalty update rule
  and
  \codeName{ns1208400}, \codeName{triptim1}, \codeName{acc-tight5},
  \codeName{enlight13}, \codeName{ex9}, and \codeName{ns1758913} are
  solved to global optimality using $\inc_{\text{a}}$.}
and that both parameterizations of our algorithm find a feasible
solution for approximately \SI{90}{\percent} of the \MIPLIBtwentyten
benchmark instances (75 out of 84~instances for the $\inc_{\text{a}}$
update rule and 76 for the multiplicative rule~$\inc_{\text{m}}$).
Moreover, the multiplicative update operator~$\inc_{\mathrm{m}}$
yields a slightly more robust algorithm, \ie it finds a feasible
solution for a few more instances than the additive
version~$\inc_{\mathrm{a}}$.
The right part of \reffig{fig:miplib2010-winners} compares the two
winner instantiations \wrt running times.
It can be seen that the additive $\inc_{\mathrm{a}}$ update operator tends to
result in a faster algorithm for significantly more instances than
the multiplicative version (\SI{59.5}{\percent}
vs. \SI{38.1}{\percent}).

Next, we compare our penalty \adm based feasibility pump with the
objective feasibility pump of \textcite{Achterberg_Berthold:2007}.
In order to achieve a fair comparison, we extend our method with
a local branching strategy with an additional $k$-opt neighborhood
constraint with $k = \min \set{20, \lfloor \card{\intIdxSet}/10
  \rfloor}$ and a time limit~$t^+ - t^*$;
\cf~\textcite{Fischetti_Lodi:2003}.
Here $t^*$ denotes the time spent in the penalty ADM based feasibility
pump itself.
Thus, the local branching stage serves as an improvement
heuristic as it is also the case in~\cite{Achterberg_Berthold:2007}.
\begin{figure}[tp]
  \centering
  \includegraphics[width=0.65\textwidth]{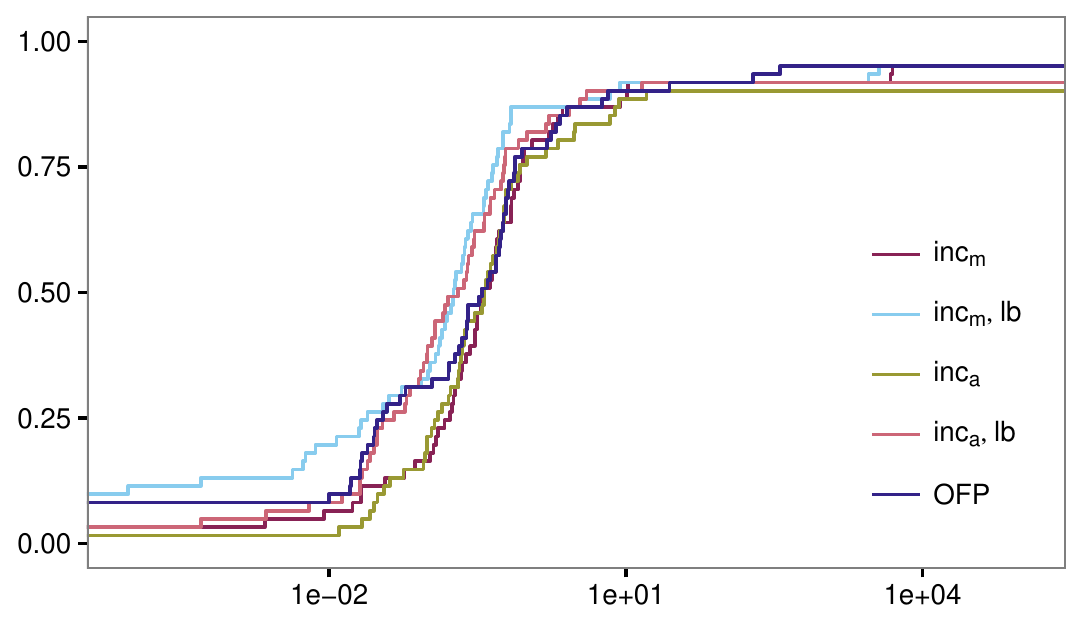}
  \caption{The two winner parameterizations
    (\cf~\reffig{fig:miplib2010-winners}) with and without local
    branching compared to the objective feasibility pump (OFP) by
    \textcite{Achterberg_Berthold:2007}}
  \label{fig:miplib2010-comp-w-ofp}
\end{figure}
As test instances we use all \MIPLIBtwentyten and \MIPLIBtwentothree
instances that have been used in~\cite{Achterberg_Berthold:2007} as
well.
\reffig{fig:miplib2010-comp-w-ofp} shows the primal-dual gap
performance profiles of the two winner parameterizations ($\lambda =
0.9$ and $\inc \in \set{\inc_{\mathrm{m}}, \inc_{\mathrm{a}}}$) with
and without local branching applied as an additional improvement
heuristic as well as the corresponding performance profile curve based
on the results reported by~\textcite{Achterberg_Berthold:2007}.
First of all, it can be seen that all five methods find a feasible
solution for at least \SI{90.2}{\percent} of the tested instances, which
underpins the strength of feasibility pumps in general.
Comparing only the different parameterization of our method we see
that the $\inc_{\text{m}}$ update rule with local branching
outperforms the version without local branching and both variants
using the additive penalty update rule.
The latter also performs similar independent of whether local
branching is used or not, whereas the local branching stage
significantly improves the solution quality when the $\inc_{\text{m}}$
rule is used (in which case we find a global optimal solution for
\SI{11.5}{\percent}; compared to \SI{8.2}{\percent} for the objective
feasibility pump of~\textcite{Achterberg_Berthold:2007}).
One sees that the multiplicative update rule together
with local branching slightly outperforms the objective feasibility
pump of \textcite{Achterberg_Berthold:2007} \wrt solution quality.

\begin{figure}
  \centering
  \includegraphics[width=0.48\textwidth]{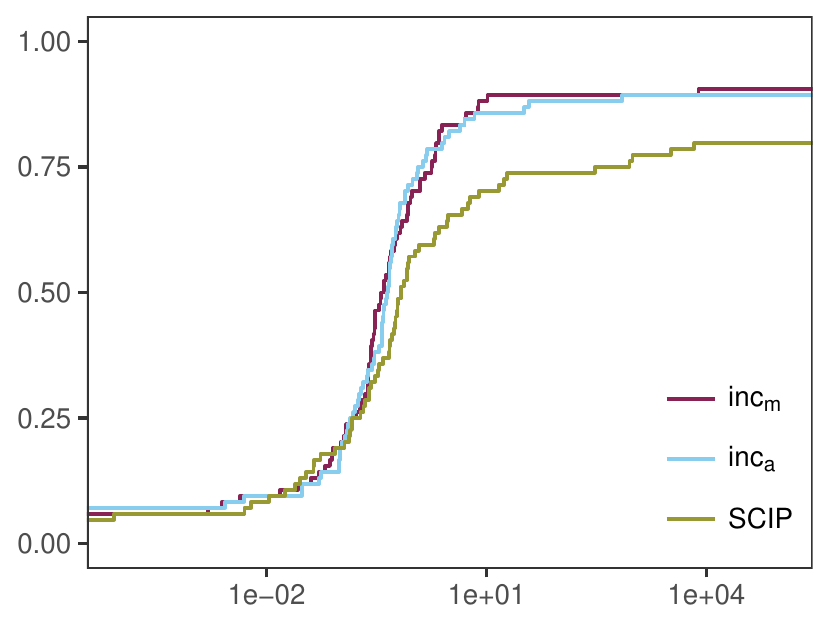}%
  \quad
  \includegraphics[width=0.48\textwidth]{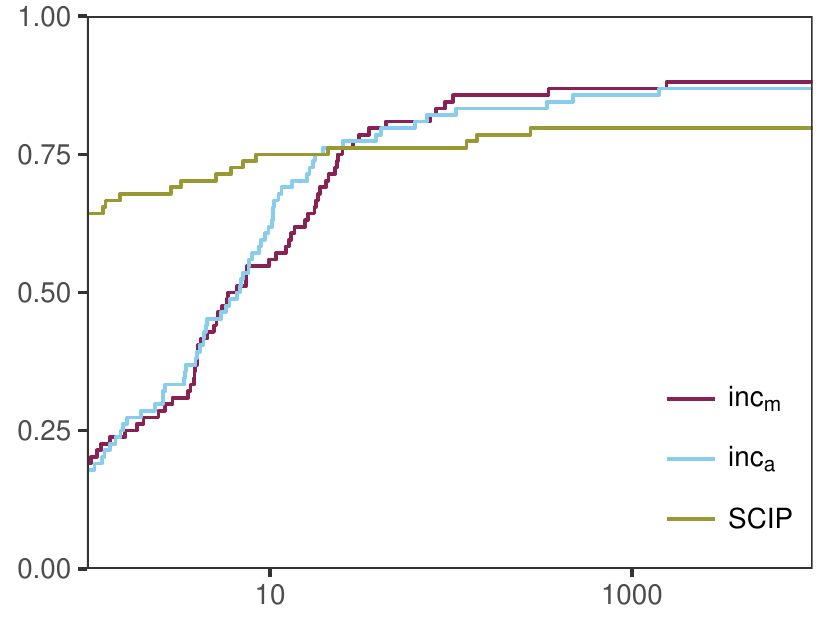}%
  \caption{The two winner parameterizations (see Figure 2) with and
    without local branching compared to the feasibility pump
    implementation for MIPs of \SCIP: primal dual gap (left) and
    running times (right)}
  \label{fig:scip-fp-comparison}
\end{figure}
We now turn to a comparison with the MIP solver \SCIP~3.2.1.
We used \SCIP instead of, \eg \Cplex or \Gurobi, for our comparison
with a state-of-the-art MIP solver because \SCIP is the only solver
that allows to completely de-activate all other components of the solution
process such that we can compare solution quality and running times.
To this end, we de-activated \SCIP's presolve, all cuts, and all
heuristics except for the feasibility pump.\footnote{The feasibility
  pump specific \SCIP parameters are
  \texttt{heuristics/feaspump/maxloops=-1},
  \texttt{heuristics/feaspump/maxlpiterofs=2147483647},
  \texttt{heuristics/feaspump/maxlpiterquot=1e10}, and
  \texttt{heuristics/feaspump/maxstallloops=-1}.
  They are used to avoid a too early stopping of \SCIP's feasibility
  pump.}
We also use a time limit of \SI{1}{\hour}, stop the algorithm
after finding the first feasible solution, and use \Cplex~12.6 as the
internal LP solver.
To be comparable with our feasibility pump implementation that uses
\Gurobi's preprocessing, we presolved all 84~\MIPLIBtwentyten
instances with \Gurobi and solved these presolved instances with
\SCIP's feasibility pump implementation.
The results are given in \reffig{fig:scip-fp-comparison}.
The left figure shows the performance profile of the primal-dual gap.
We see that we are again comparable in terms of solution
quality and that our solutions tend to have a slightly better
objective value.
Moreover, we find a feasible solution for up to approximately
\SI{90}{\percent} of all instances, whereas \SCIP finds a feasible
point for slightly more than \SI{75}{\percent}.
However, this comes at the price of significantly larger running
times; \cf \reffig{fig:scip-fp-comparison} (right).
\SCIP is faster for almost all instances and solves most of them
within approximately \SI{10}{\second}.
Although we did not try to tune our code extensively so far,
the latter comparison shows that there is still a strong potential and
a lot of work to do if our code should be competitive with a
state-of-the-art heuristic \wrt running times.

\begin{figure}
  \centering
  \includegraphics[width=0.48\textwidth]{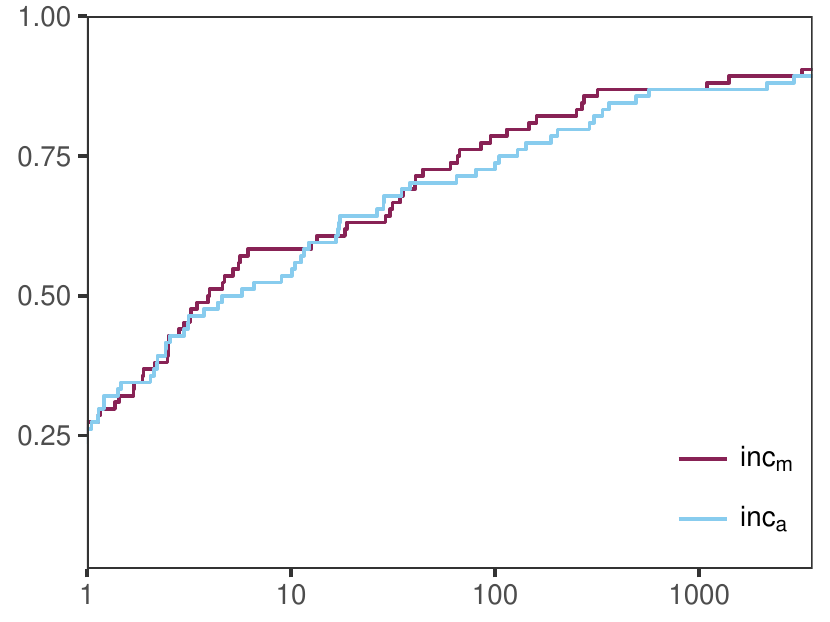}
  \quad
  \includegraphics[width=0.48\textwidth]{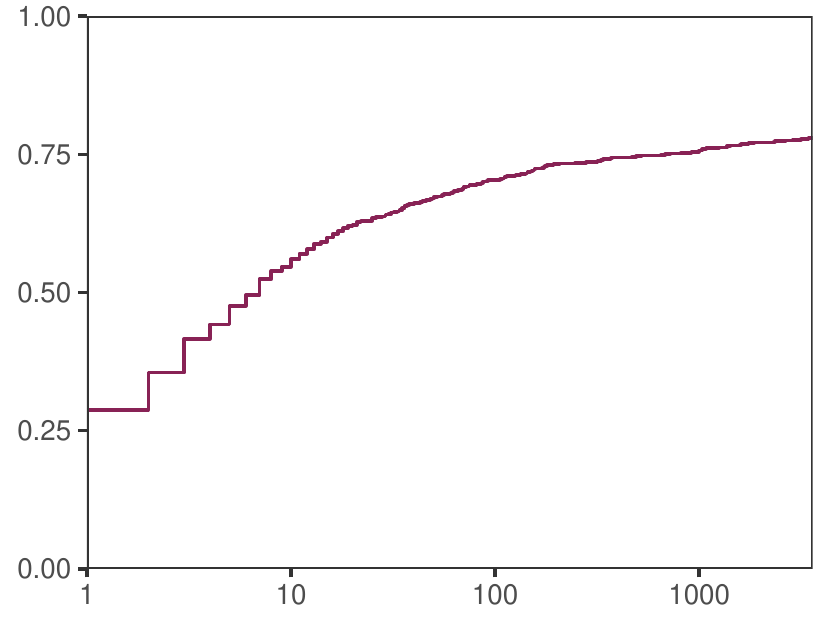}
  \vspace*{-1em}
  \caption{Left: Cumulative distribution function of absolute running
    times ($x$-axis; in s) for two winner parameterizations
    $\inc \in \set{\inc_{\mathrm{m}}, \inc_{\mathrm{a}}}$ on all \MIPLIBtwentyten instances.
    Right: Cumulative distribution function of absolute running
    times ($x$-axis; in s) on all \MINLPLibTwo instances}
  \label{fig:ecdf-absolute-times}
\end{figure}

In \reffig{fig:ecdf-absolute-times} (left) cumulative distribution
functions for absolute running times are given for both winner
parameterizations of our algorithm for \mips.
It can be seen that for approximately \SI{25}{\percent} of all instances
a feasible solution is found in at most \SI{1}{\second} and
\SIrange{50}{60}{\percent} of all instances have been solved to
feasibility within \SI{10}{\second}.
The geometric mean of the running times taken over all instances for
which a feasible solution has been found within the time limit is
\SI{4.56}{\second} for $\inc = \inc_{\mathrm{m}}$ and
\SI{4.94}{\second} for $\inc = \inc_{\mathrm{a}}$.
\ifPreprint
The running times for all instances can be found in
Appendix~\ref{sec:detailed-results-miplib2010}.
\fi

\begin{figure}[tp]
  \centering
  \includegraphics[width=0.65\textwidth]{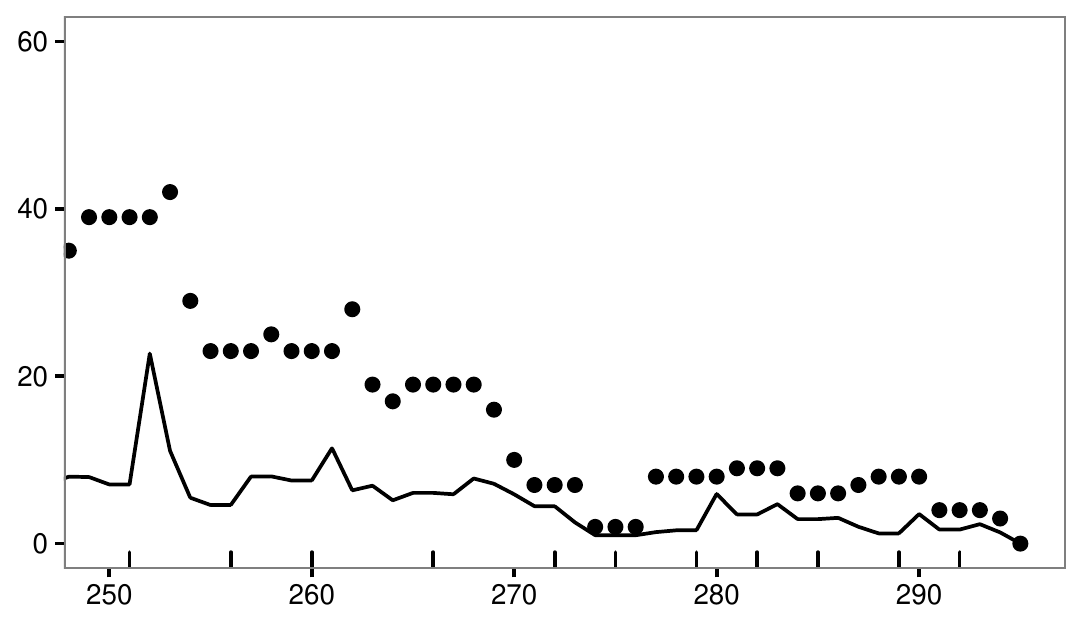}
  \caption{Number of fractional integer components (dots) and total
    fractionality (solid line) vs. ADM iterations.
    Penalty parameter updates are marked with small black vertical
    lines on top of the ADM iteration axis.}
  \label{fig:rococoC10-001000-iter-log}
\end{figure}
We close this section with an exemplary discussion of the course of
integer (in)feasibility during the iterations of our method.
\reffig{fig:rococoC10-001000-iter-log} shows the approximately 50~last
iterations of our method applied to the \MIPLIBtwentyten instance
\codeName{rococoC10-001000}.
Dots correspond to numbers of fractional integer components and the
solid line represents the course of the total fractionality measure
\begin{equation*}
  \sum_{\intIdx \in \intIdxSet} \abs{x_\intIdx - \lfloor x_\intIdx
  + 0.5 \rfloor}
\end{equation*}
of solutions $x$ of the continuous subproblems over the subsequent ADM
iterations.
The small black lines on top of the ADM iteration axis denote
iterations at which the penalty parameters are updated.

First of all, we see that penalty parameter updates are applied whenever
the ADM of the inner loop stalls, \ie whenever the ADM of the inner
loop entered an undesired integer infeasible partial minimum.
The method stops after 295~ADM iterations with an integer feasible partial
minimum.
As expected, we typically see a sawtooth phenomenon:
The total fractionality decreases between two consecutive
penalty parameter updates and increases after a penalty
parameter update.
The number of fractional integer components follows this behavior
qualitatively.
The number of ADM iterations between two consecutive penalty parameter
updates varies between 3 to 6 iterations.
Thus, convergence to partial minima does not seem to be challenging
for this specific instance.


\subsection{Mixed-Integer Nonlinear Problems}
\label{sec:comp-results:minlp}

We now turn to \mixedinteger \nonlinear programs.  The penalty based
ADM for this class of models has been implemented in \Cpp using the
so-called \GAMS Expert-level API with \GAMS~24.5.4
\cite{GamsSoftware2015}.
The continuous relaxation models are solved with
\CONOPT~3.17A~\cite{ConoptSoftware}.
According to the results from \refsec{sec:comp-results:mip} we
choose the parameters $\alpha^0=1$ and $\lambda=0.9$.
The penalty parameter update rule is chosen to be
$\inc_{\mathrm{a}}(x) = x+1$ since this variant turned out to be
favorable for \minlps.
We set the time limit to \SI{1}{\hour} as for the \mip experiments and
we do not incorporate any iteration limits for the inner ADM and the
outer penalty loop.
Throughout this section we declare an \minlp instance as solved to
feasibility if \CONOPT finds a feasible solution (\wrt its default
tolerances) with all integer components fixed to integral values.

We compare the results of our method with other recently
published numerical results concerning feasibility pump algorithms for
convex and \nonconvex \minlps.
Most of the results from the literature that we use for these
comparisons are based on the first and second version of the
\MINLPLib; \cf~\textcite{Bussieck_et_al:2003}.
Since a reasonable comparison of running times is not possible due to
differences in the used hardware, we
focus on the comparison of success rates and solution quality.

We also compared the obtained results separately for convex and
\nonconvex instances.
As expected, the results of our method are slightly better for
the convex case.
However, the results are qualitatively rather similar and we thus
present the following analysis of our computational results
without distinguishing between convex and \nonconvex instances.

First, we start with a comparison of different feasibility pump versions
presented
\ifPreprint
by~\citeauthor{DAmbrosio_et_al:2012} in~\cite{DAmbrosio_et_al:2012}
\fi
and our method on selected \MINLPLib instances.
\ifPreprint
\citeauthor{DAmbrosio_et_al:2012}
\fi
test their method on 65~\MINLPLib instances.
However, it turned out in the meantime that the used test set
contained 4~duplicate instances.
Thus, the following comparison is carried out on 61~\MINLPLib instances.
\reffig{fig:minlplib-primal-gaps-comp} (left) displays the
performance profiles using the primal-dual gap as the performance
measure; \cf~\eqref{eq:primal-dual-gap-def}.
\begin{figure}[tp]
  \centering
  \includegraphics[width=0.48\textwidth]{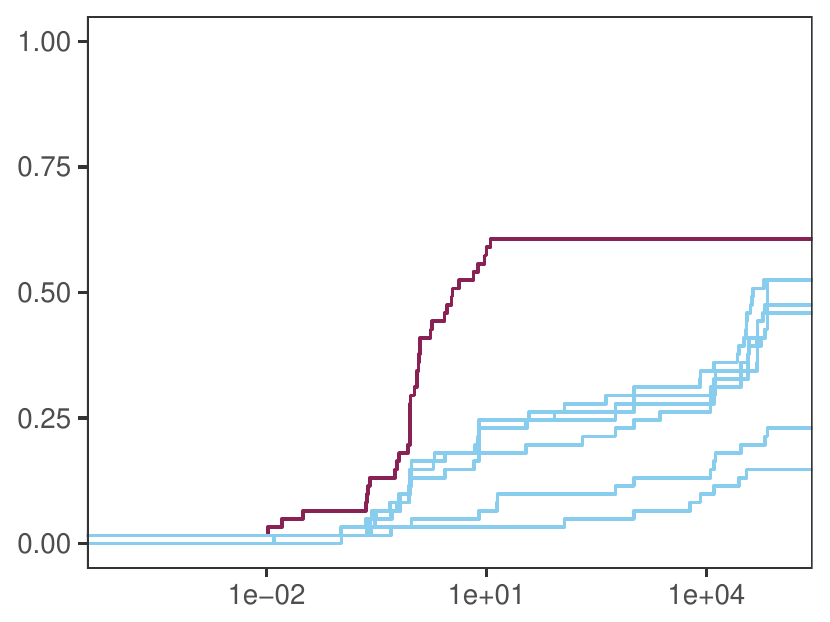}
  \quad
  \includegraphics[width=0.48\textwidth]{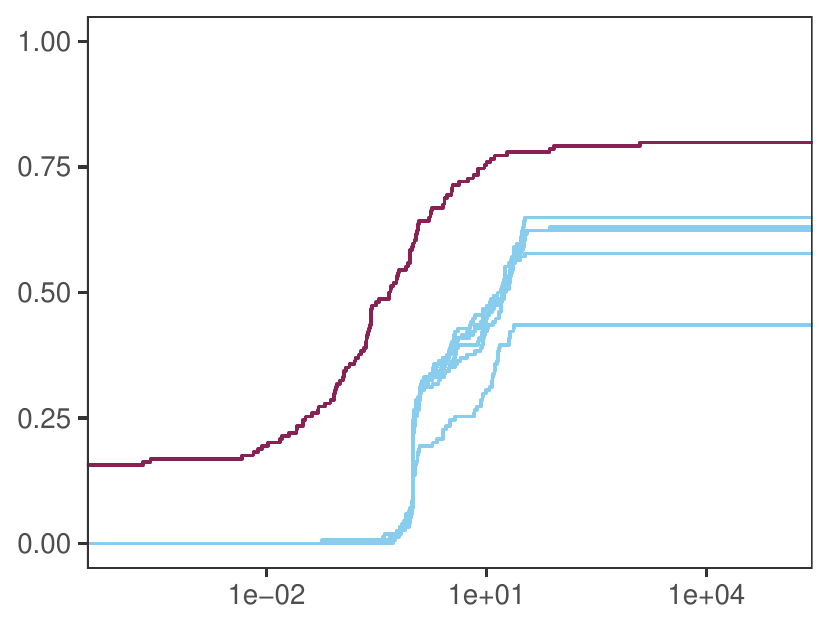}
  \vspace*{-1em}
  \caption{Left: Performance profiles for the primal-dual gap of the penalty
    ADM based feasibility pump (red) and all six feasibility pump
    variants (blue) for \minlps presented
    in~\cite{DAmbrosio_et_al:2012} on subset of 61~\MINLPLib
    instances.
    Right: Performance profiles for the primal-dual gap of the penalty
    ADM based feasibility pump (red) and different feasibility pump
    variants (blue) proposed
    \ifPreprint
    by~\textcite{Berthold:2014}
    \fi
    }
  \label{fig:minlplib-primal-gaps-comp}
\end{figure}
It can be clearly seen that the penalty ADM based feasibility pump
outperforms all feasibility pump variants presented
in~\cite{DAmbrosio_et_al:2012},
although we used a time limit of \SI{1}{\hour} in contrast to
\SI{2}{\hour} used in \cite{DAmbrosio_et_al:2012}.
Although the number of instances solved to optimality is comparably
low for all algorithms, the overall solution quality of our penalty
ADM based algorithm is significantly higher than the quality of
solutions obtained in~\textcite{DAmbrosio_et_al:2012}.
Additionally, the number of instances for which we found a feasible
solution is \SI{60.7}{\percent} whereas the percentage
of instances solved to feasibility in
\textcite{DAmbrosio_et_al:2012} ranges from
\SI{54.1}{\percent} to only \SI{16.4}{\percent}.

Complementing this comparison on the \MINLPLib test set we also
tested our method on 889 out of 1385~instances of the more recent
\MINLPLibTwo test set.
Here, we neglected all instances containing only continuous
variables.
Again, our method behaves quite satisfactory.
It computes a feasible solution for 642~instances
(\SI{72.2}{\percent}), leaving 247~instances unsolved.
We compare our solutions with the dual bound given in the
\MINLPLibTwo.
The penalty ADM based feasibility pump computes solutions
with a vanishing primal-dual gap for 83~instances;
\ie for \SI{9.3}{\percent} of all instances of the test set.
The running times on all \MINLPLibTwo instances are given in the
cumulative distribution function in
\reffig{fig:ecdf-absolute-times} (right).
Remarkably, the results are quite comparable with those for the \mip
instances:
For approximately \SI{50}{\percent} of all \MINLPLibTwo instances a
feasible solution has been found within \SI{10}{\second}
and slightly more than \SI{25}{\percent} of the instances are solved
to feasibility within \SI{1}{\second}.
The geometric mean of the running times taken over all instances for
which a feasible solution has been found within the time limit is
\SI{5.36}{\second}.
Nevertheless, the running times are fast enough so that the method
could, in principle, be embedded in a global solver as a primal heuristic.
\ifPreprint
As for the \mip results, a table with all running times and additional
information is given in the appendix.
\fi

Finally, we discuss the most recent (at least to the best of our
knowledge) results on a new variant
of feasibility pumps for \nonconvex \mixedinteger \nonlinear problems
that are published
\ifPreprint
by~\citeauthor{Berthold:2014} in his PhD thesis~\cite{Berthold:2014}.
\fi
The test set used
\ifPreprint
by~\citeauthor{Berthold:2014}
\fi
is neither a sub- nor
a superset of the current \MINLPLib.
Thus, we compare our method with the algorithms proposed
\ifPreprint
by~\citeauthor{Berthold:2014}
\fi
on all instances of his test set that
are also part of the current \MINLPLib version,
yielding a test set of 154~instances.
Again, we compare the methods by performance profiles of the
primal-dual gap,
\cf~\reffig{fig:minlplib-primal-gaps-comp} (right).
The penalty ADM based feasibility pump significantly
outperforms all variants of the feasibility pump for \minlps proposed
\ifPreprint
by~\citeauthor{Berthold:2014}.
\fi
First, the methods
\ifPreprint
of~\citeauthor{Berthold:2014}
\fi
do not solve any instance to optimality, whereas we close the
primal-dual gap for \SI{15.6}{\percent} of the instances.
Second, our method finds a feasible solution for
\SI{79.9}{\percent} of the tested instances, whereas the methods
\ifPreprint
of~\citeauthor{Berthold:2014}
\fi
solve approximately \SIrange{43.5}{64.9}{\percent} to feasibility.




\section{Summary}
\label{sec:conclusion}

In this paper we have shown that idealized feasibility pumps, \ie
feasibility pumps without random perturbations, can be seen as
alternating direction methods applied to a special reformulation of
the original \mixedinteger problem.
This yields that idealized feasibility pumps converge to a partial
minimum of the reformulated problem.
If this partial minimum is not an integer feasible point,
feasibility pumps apply a random perturbation to escape this undesired
point.
We replace this random restart with a penalty
framework that encompasses the alternating direction method in the
inner loop and that replaces random perturbations by tailored penalty
parameter updates.
This way it is possible for the first time to perform a theoretical
study for a variant of the feasibility pump including restarts.
The resulting penalty based alternating direction method can be
applied to both \mips and \minlps.
Our numerical results indicate that this new version of the
feasibility pump is comparable (\wrt most recent publications) for
the case of \mips and clearly outperforms other feasibility pump
algorithms on \minlps in terms of solution quality.


\section*{Acknowledgements}%
\label{sec:acknowledgements}%
We acknowledge funding through the DFG SFB/Transregio~154,
Subprojects~A05, B07, and B08.
This research has been performed as part of the Energie Campus
N\"urnberg and supported by funding through the ``Aufbruch Bayern
(Bavaria on the move)'' initiative of the state of Bavaria.


\printbibliography
\newpage
\appendix
\section{Detailed Results for the \textsf{MIPLIB~2010}}
\label{sec:detailed-results-miplib2010}

\begin{longtable}{lrrrr}
\caption{Detailed numerical results for the penalty ADM based feasibility pump on all 84~feasible \MIPLIBtwentyten instances; \cf~Section 5.1 Mixed-Integer Linear Problems. Objective function value, running times (s), outer penalty iterations (\#Pen.), and inner ADM iterations (\#ADM) for the method with $\lambda = 0.9$, multiplicative penalty parameter update rule~$\inc_{\text{m}}$, activated \Gurobi presolve, and deactivated local branching.} \\ 
  \toprule
Instance & Objective & Time & \#Pen. & \#ADM \\ 
  \midrule
\endhead
\bottomrule
\endfoot
30n20b8 & 906.00 & 40.42 & 892 & 3828 \\ 
  acc-tight5 & --- & --- & --- & --- \\ 
  aflow40b & 3565.00 & 2.48 & 146 & 437 \\ 
  air04 & 71018.00 & 60.42 &  41 & 190 \\ 
  app1-2 & -30.00 & 3189.86 & 4133 & 15567 \\ 
  bab5 & -78587.20 & 34.25 & 178 & 540 \\ 
  beasleyC3 & 863.00 & 0.19 &  24 &  67 \\ 
  biella1 & 3743460.00 & 65.50 &  26 & 238 \\ 
  bienst2 & 73.25 & 0.21 &  20 &  76 \\ 
  binkar10\_1 & 7263.57 & 0.23 &  32 & 112 \\ 
  bley\_xl1 & 285.00 & 4.62 &  37 & 149 \\ 
  bnatt350 & --- & --- & --- & --- \\ 
  core2536-691 & 692.00 & 31.37 &  19 & 106 \\ 
  cov1075 & 120.00 & 0.17 &  12 &  24 \\ 
  csched010 & --- & --- & --- & --- \\ 
  danoint & 78.00 & 1.70 &  24 & 101 \\ 
  dfn-gwin-UUM & 125512.00 & 0.28 &  66 & 166 \\ 
  eil33-2 & 1373.60 & 0.74 &  19 &  64 \\ 
  eilB101 & 1513.00 & 0.94 &  17 &  66 \\ 
  enlight13 & 0.00 & 0.00 &   1 &   1 \\ 
  ex9 & 0.00 & 0.00 &   1 &   1 \\ 
  glass4 & 3390030000.00 & 0.28 & 103 & 353 \\ 
  gmu-35-40 & -2159090.00 & 1.89 & 298 & 1073 \\ 
  iis-100-0-cov & 100.00 & 0.27 &  12 &  24 \\ 
  iis-bupa-cov & 100.00 & 2.16 &  12 &  38 \\ 
  iis-pima-cov & 74.00 & 3.23 &  12 &  44 \\ 
  lectsched-4-obj & 9.00 & 3.46 &  66 & 287 \\ 
  m100n500k4r1 & -20.00 & 0.76 &  30 & 203 \\ 
  macrophage & 522.00 & 0.32 &  14 &  30 \\ 
  map18 & -280.00 & 250.17 & 120 & 317 \\ 
  map20 & -371.00 & 159.76 & 127 & 306 \\ 
  mcsched & 228737.00 & 1.37 &  13 &  56 \\ 
  mik-250-1-100-1 & 284980.00 & 0.48 &  71 & 202 \\ 
  mine-166-5 & -22751900.00 & 1.16 &  22 &  54 \\ 
  mine-90-10 & -558839000.00 & 1.69 &  58 & 173 \\ 
  msc98-ip & 25797700.00 & 146.92 &  48 & 307 \\ 
  mspp16 & 407.00 & 1399.15 &  91 & 249 \\ 
  mzzv11 & -16438.00 & 94.84 & 102 & 440 \\ 
  n3div36 & 170600.00 & 13.36 & 103 & 321 \\ 
  n3seq24 & 53600.00 & 271.77 &  71 & 315 \\ 
  n4-3 & 13980.00 & 0.28 &  13 &  35 \\ 
  neos-1109824 & 687.00 & 3.93 &  93 & 295 \\ 
  neos13 & -28.04 & 5.54 &  23 &  91 \\ 
  neos-1337307 & -201818.00 & 2.99 &  29 & 100 \\ 
  neos-1396125 & --- & --- & --- & --- \\ 
  neos-1601936 & 23409.00 & 266.56 &  66 & 629 \\ 
  neos18 & 17.00 & 0.87 &  21 &  86 \\ 
  neos-476283 & 407.01 & 66.98 &  78 & 226 \\ 
  neos-686190 & 16410.00 & 1.87 &  47 & 143 \\ 
  neos-849702 & --- & --- & --- & --- \\ 
  neos-916792 & 44.56 & 5.64 & 100 & 335 \\ 
  neos-934278 & 264.00 & 1091.80 &  57 & 225 \\ 
  net12 & 337.00 & 29.06 &  53 & 217 \\ 
  netdiversion & --- & --- & --- & --- \\ 
  newdano & 89.75 & 0.44 &  19 &  77 \\ 
  noswot & -38.00 & 0.17 &  61 & 240 \\ 
  ns1208400 & --- & --- & --- & --- \\ 
  ns1688347 & 35.00 & 6.14 &  49 & 227 \\ 
  ns1758913 & -1454.67 & 18.37 &   5 &  15 \\ 
  ns1830653 & --- & --- & --- & --- \\ 
  opm2-z7-s2 & -1519.00 & 35.30 &  13 &  37 \\ 
  pg5\_34 & -12628.50 & 0.76 &  65 & 195 \\ 
  pigeon-10 & -9000.00 & 2.51 & 1091 & 2732 \\ 
  pw-myciel4 & 13.00 & 3.99 &  20 & 130 \\ 
  qiu & 1235.01 & 0.44 &  19 &  51 \\ 
  rail507 & 183.00 & 30.66 &  35 & 158 \\ 
  ran16x16 & 4734.00 & 0.12 &  52 & 144 \\ 
  reblock67 & -18739300.00 & 1.14 &  51 & 169 \\ 
  rmatr100-p10 & 494.00 & 2.49 &  12 &  41 \\ 
  rmatr100-p5 & 1327.00 & 5.18 &  11 &  39 \\ 
  rmine6 & -239.11 & 3.21 &  32 &  93 \\ 
  rocII-4-11 & -0.52 & 18.74 & 292 & 1162 \\ 
  rococoC10-001000 & 34598.00 & 1.43 &  68 & 295 \\ 
  roll3000 & 18404.00 & 2.83 &  48 & 207 \\ 
  satellites1-25 & 33.00 & 44.26 &  24 &  98 \\ 
  sp98ic & 558066000.00 & 12.55 &  84 & 313 \\ 
  sp98ir & 244287000.00 & 4.73 &  49 & 168 \\ 
  tanglegram1 & 6478.00 & 114.05 &  16 &  35 \\ 
  tanglegram2 & 1445.00 & 2.51 &  16 &  34 \\ 
  timtab1 & 1415540.00 & 0.27 &  41 & 153 \\ 
  triptim1 & 22.87 & 317.56 &   7 &  16 \\ 
  unitcal\_7 & 20426600.00 & 40.69 & 112 & 400 \\ 
  vpphard & 44.00 & 85.31 &  38 & 242 \\ 
  zib54-UUE & 13164400.00 & 0.44 &  27 &  75 \\ 
  \end{longtable}


\section{Detailed Results for the \textsf{MINLPLib2}}
\label{sec:detailed-results-minlplib}

\begin{longtable}{lrrrr}
\caption{Detailed numerical results for the penalty ADM based feasibility pump on all 889 \MINLPLib2 instances with integer variables. Objective function value, running times (s), outer penalty iterations (\#Pen.), and inner ADM iterations (\#ADM) for the method with $\lambda = 0.9$, additive penalty parameter update rule~$\inc_{\text{a}}$, and deactivated local branching.} \\ 
  \toprule
Instance & Objective & Time & \#Pen. & \#ADM \\ 
  \midrule
\endhead
\bottomrule
\endfoot
4stufen & 118114.00 &   9 &  74 & 169 \\ 
  alan & 3.00 &   0 &   5 &   7 \\ 
  autocorr\_bern20-03 & -64.00 &   0 &   1 &   1 \\ 
  autocorr\_bern20-05 & -396.00 &   2 &   1 &   1 \\ 
  autocorr\_bern20-10 & -2912.00 &   1 &   1 &   1 \\ 
  autocorr\_bern20-15 & -5936.00 &   1 &   1 &   1 \\ 
  autocorr\_bern25-03 & -80.00 &   0 &   1 &   1 \\ 
  autocorr\_bern25-06 & -936.00 &   0 &   3 &   4 \\ 
  autocorr\_bern25-13 & -7984.00 &   0 &   1 &   1 \\ 
  autocorr\_bern25-19 & -14472.00 &   0 &   1 &   1 \\ 
  autocorr\_bern25-25 & -10352.00 &   0 &   1 &   1 \\ 
  autocorr\_bern30-04 & -288.00 &   0 &   1 &   1 \\ 
  autocorr\_bern30-08 & -2912.00 &   0 &   1 &   1 \\ 
  autocorr\_bern30-15 & -15384.00 &   0 &   1 &   1 \\ 
  autocorr\_bern30-23 & -30240.00 &   1 &   1 &   1 \\ 
  autocorr\_bern30-30 & -22640.00 &   1 &   1 &   1 \\ 
  autocorr\_bern35-04 & -344.00 &   0 &   1 &   1 \\ 
  autocorr\_bern35-09 & -4976.00 &   0 &   1 &   1 \\ 
  autocorr\_bern35-18 & -30712.00 &   0 &   1 &   1 \\ 
  autocorr\_bern35-26 & -54960.00 &   0 &   1 &   1 \\ 
  autocorr\_bern35-35 & -40272.00 &   1 &   1 &   1 \\ 
  autocorr\_bern40-05 & -908.00 &   0 &   1 &   1 \\ 
  autocorr\_bern40-10 & -8192.00 &   0 &   1 &   1 \\ 
  autocorr\_bern40-20 & -50228.00 &   1 &   5 &   7 \\ 
  autocorr\_bern40-30 & -94040.00 &   1 &   1 &   1 \\ 
  autocorr\_bern40-40 & -66832.00 &   1 &   1 &   1 \\ 
  autocorr\_bern45-05 & -1004.00 &   0 &   3 &   4 \\ 
  autocorr\_bern45-11 & -12532.00 &   0 &   1 &   1 \\ 
  autocorr\_bern45-23 & -84844.00 &   0 &   1 &   1 \\ 
  autocorr\_bern45-34 & -151768.00 &   0 &   1 &   1 \\ 
  autocorr\_bern45-45 & -108528.00 &   1 &   1 &   1 \\ 
  autocorr\_bern50-06 & -2072.00 &   1 &   3 &   4 \\ 
  autocorr\_bern50-13 & -23176.00 &   0 &   1 &   1 \\ 
  autocorr\_bern50-25 & -123764.00 &   1 &   3 &   4 \\ 
  autocorr\_bern50-38 & -232808.00 &   0 &   1 &   1 \\ 
  autocorr\_bern50-50 & -166168.00 &   1 &   1 &   1 \\ 
  autocorr\_bern55-06 & -2288.00 &   0 &   3 &   4 \\ 
  autocorr\_bern55-14 & -32280.00 &   1 &   3 &   4 \\ 
  autocorr\_bern55-28 & -189404.00 &   0 &   1 &   1 \\ 
  autocorr\_bern55-41 & -335980.00 &   0 &   1 &   1 \\ 
  autocorr\_bern55-55 & -238296.00 &   1 &   1 &   1 \\ 
  autocorr\_bern60-08 & -6712.00 &   0 &   1 &   1 \\ 
  autocorr\_bern60-15 & -44368.00 &   0 &   1 &   1 \\ 
  autocorr\_bern60-30 & -258304.00 &   0 &   1 &   1 \\ 
  autocorr\_bern60-45 & -476456.00 &   1 &   1 &   1 \\ 
  autocorr\_bern60-60 & -347372.00 &   1 &   1 &   1 \\ 
  batch & 309205.00 &   2 &  22 &  43 \\ 
  batch0812 & 2838520.00 &   5 &  64 & 107 \\ 
  batch0812\_nc & 3534900.00 &   4 &  32 &  70 \\ 
  batchdes & 185769.00 &   0 &   7 &  13 \\ 
  batch\_nc & 363583.00 &   5 &  60 & 105 \\ 
  batchs101006m & 776397.00 &   5 &  24 &  65 \\ 
  batchs121208m & 1336450.00 &   7 &  22 &  73 \\ 
  batchs151208m & 1588930.00 &  12 &  37 & 120 \\ 
  batchs201210m & 2408440.00 &  17 &  72 & 157 \\ 
  bchoco05 & 0.95 &   1 &   3 &   4 \\ 
  bchoco06 & 0.96 &   0 &   3 &   4 \\ 
  bchoco07 & 0.96 &   2 &   3 &   4 \\ 
  bchoco08 & --- & --- & --- & --- \\ 
  beuster & 128512.00 &  15 &  99 & 254 \\ 
  blend029 & --- & --- & 68919 & 68961 \\ 
  blend146 & 29.93 &   8 &  76 & 110 \\ 
  blend480 & -8.24 &  56 & 655 & 753 \\ 
  blend531 & --- & --- & 53222 & 53555 \\ 
  blend718 & 1.23 &   6 &  30 &  77 \\ 
  blend721 & -0.87 &   4 &  27 &  56 \\ 
  blend852 & 46.13 &   9 &  87 & 126 \\ 
  blendgap & -0.00 &   0 &   1 &   1 \\ 
  cardqp\_inlp & 3843.61 &   0 &   1 &   1 \\ 
  cardqp\_iqp & 3843.61 &   0 &   1 &   1 \\ 
  carton7 & 303.88 & 368 & 4645 & 5010 \\ 
  carton9 & 340.75 &  15 &  77 & 155 \\ 
  casctanks & 9.16 &   3 &  40 &  46 \\ 
  case\_1scv2 & 7791.24 &  16 &  33 &  86 \\ 
  cecil\_13 & -115564.00 &   9 &  65 &  92 \\ 
  chp\_partload & --- & --- & 25492 & 25701 \\ 
  clay0203h & 41709.80 &  81 & 1498 & 1519 \\ 
  clay0203m & 41737.50 &  50 & 1002 & 1022 \\ 
  clay0204h & 7830.00 &   5 &  50 &  82 \\ 
  clay0204m & 10340.00 &   3 &  35 &  48 \\ 
  clay0205h & 23484.20 &  43 & 609 & 715 \\ 
  clay0205m & 9715.00 &   5 &  57 &  86 \\ 
  clay0303h & 36613.00 &  56 & 954 & 996 \\ 
  clay0303m & 41737.50 &   8 & 101 & 170 \\ 
  clay0304h & 61315.90 &  52 & 764 & 859 \\ 
  clay0304m & 61831.50 &  38 & 694 & 748 \\ 
  clay0305h & 43405.70 &  10 & 132 & 161 \\ 
  clay0305m & 24593.10 &   7 &  92 & 121 \\ 
  contvar & 829318.00 &   7 &   8 &  35 \\ 
  crossdock\_15x7 & 14467.00 &   6 &  62 &  99 \\ 
  crossdock\_15x8 & 16765.00 & 1856 & 26313 & 27298 \\ 
  crudeoil\_lee1\_05 & 79.35 &  75 & 808 & 851 \\ 
  crudeoil\_lee1\_06 & 78.75 &  70 & 555 & 655 \\ 
  crudeoil\_lee1\_07 & 78.75 & 2868 & 24641 & 25051 \\ 
  crudeoil\_lee1\_08 & 79.35 &  18 &  72 & 106 \\ 
  crudeoil\_lee1\_09 & 78.75 & 918 & 5556 & 5841 \\ 
  crudeoil\_lee1\_10 & 78.75 &  32 & 103 & 153 \\ 
  crudeoil\_lee2\_05 & 90.00 &  31 & 117 & 158 \\ 
  crudeoil\_lee2\_06 & 97.59 & 114 & 265 & 401 \\ 
  crudeoil\_lee2\_07 & 90.00 & 329 & 949 & 1114 \\ 
  crudeoil\_lee2\_08 & --- & --- & 12801 & 13253 \\ 
  crudeoil\_lee2\_09 & --- & --- & --- & --- \\ 
  crudeoil\_lee2\_10 & --- & --- & 7412 & 8007 \\ 
  crudeoil\_lee3\_05 & 82.00 & 173 & 996 & 1033 \\ 
  crudeoil\_lee3\_06 & 84.49 & 125 & 574 & 606 \\ 
  crudeoil\_lee3\_07 & 82.90 &  29 &  35 &  65 \\ 
  crudeoil\_lee3\_08 & --- & --- & 13112 & 13645 \\ 
  crudeoil\_lee3\_09 & 82.75 &  66 &  85 & 108 \\ 
  crudeoil\_lee3\_10 & 77.50 & 144 & 145 & 198 \\ 
  crudeoil\_lee4\_05 & 132.48 &  13 &  15 &  21 \\ 
  crudeoil\_lee4\_06 & 132.49 &  16 &  15 &  22 \\ 
  crudeoil\_lee4\_07 & 132.55 & 170 & 272 & 333 \\ 
  crudeoil\_lee4\_08 & 131.54 & 175 & 192 & 299 \\ 
  crudeoil\_lee4\_09 & --- & --- & 5326 & 5417 \\ 
  crudeoil\_lee4\_10 & --- & --- & 4329 & 4802 \\ 
  crudeoil\_li01 & 4852.37 & 1247 & 19259 & 19368 \\ 
  crudeoil\_li02 & --- & --- & --- & --- \\ 
  crudeoil\_li03 & --- & --- & 31637 & 32220 \\ 
  crudeoil\_li05 & 3030.59 &  34 & 225 & 312 \\ 
  crudeoil\_li06 & 3303.84 & 111 & 869 & 947 \\ 
  crudeoil\_li11 & --- & --- & 26706 & 27063 \\ 
  crudeoil\_li21 & --- & --- & 19307 & 19619 \\ 
  csched1 & -30174.60 &   0 &   4 &   5 \\ 
  csched1a & -29903.30 &   1 &   7 &  13 \\ 
  csched2 & -160668.00 &   2 &   9 &  20 \\ 
  csched2a & -162047.00 &   7 &  44 & 120 \\ 
  deb10 & 209.43 &   1 &   9 &  13 \\ 
  deb6 & 251.66 &   1 &   3 &   4 \\ 
  deb7 & 176.08 &   2 &   3 &   4 \\ 
  deb8 & 176.08 &   1 &   3 &   4 \\ 
  deb9 & 176.08 &   1 &   3 &   4 \\ 
  densitymod & --- & --- &   4 &   8 \\ 
  dosemin2d & 173.98 &  11 &   5 &  11 \\ 
  dosemin3d & 1.32 &  29 &   7 &  13 \\ 
  du-opt & 5.34 &   0 &   3 &   4 \\ 
  du-opt5 & 112.02 &   0 &   5 &   8 \\ 
  edgecross10-010 & 4.00 &   0 &   1 &   1 \\ 
  edgecross10-020 & 19.00 &   0 &   1 &   1 \\ 
  edgecross10-030 & 53.00 &   0 &   3 &   4 \\ 
  edgecross10-040 & 142.00 &   0 &   3 &   4 \\ 
  edgecross10-050 & 301.00 &   1 &   1 &   1 \\ 
  edgecross10-060 & 470.00 &   0 &   1 &   1 \\ 
  edgecross10-070 & 735.00 &   0 &   1 &   1 \\ 
  edgecross10-080 & 1048.00 &   0 &   1 &   1 \\ 
  edgecross10-090 & 1387.00 &   0 &   1 &   1 \\ 
  edgecross14-019 & 5.00 &   0 &   1 &   1 \\ 
  edgecross14-039 & 123.00 &   1 &   1 &   1 \\ 
  edgecross14-058 & 391.00 &   0 &   1 &   1 \\ 
  edgecross14-078 & 725.00 &   0 &   1 &   1 \\ 
  edgecross14-098 & 1392.00 &   0 &   1 &   1 \\ 
  edgecross14-117 & 2168.00 &   1 &   3 &   4 \\ 
  edgecross14-137 & 2880.00 &   0 &   1 &   1 \\ 
  edgecross14-156 & 4342.00 &   0 &   3 &   4 \\ 
  edgecross14-176 & 5956.00 &   1 &   1 &   1 \\ 
  edgecross20-040 & 73.00 &   1 &   1 &   1 \\ 
  edgecross20-080 & 530.00 &   1 &   1 &   1 \\ 
  edgecross22-048 & 97.00 &   2 &   1 &   1 \\ 
  edgecross22-096 & 980.00 &   4 &   3 &   4 \\ 
  edgecross24-057 & 213.00 &   3 &   1 &   1 \\ 
  edgecross24-115 & 1429.00 &   3 &   1 &   1 \\ 
  eg\_all\_s & 8.67 &  11 &  20 &  44 \\ 
  eg\_disc2\_s & 5.68 &   3 &   3 &   4 \\ 
  eg\_disc\_s & 6.03 &   6 &   7 &  15 \\ 
  eg\_int\_s & 8.32 &   7 &   6 &  13 \\ 
  elf & 1.68 &   1 &   1 &   1 \\ 
  eniplac & -120713.00 &  43 & 636 & 764 \\ 
  enpro48pb & 188887.00 &   1 &  10 &  23 \\ 
  enpro56pb & 266762.00 &   3 &  24 &  56 \\ 
  ethanolh & -157.59 &   0 &   4 &   5 \\ 
  ethanolm & -31.27 &   9 & 116 & 168 \\ 
  ex1221 & 7.67 &   0 &   1 &   1 \\ 
  ex1222 & 1.08 &   0 &   6 &  10 \\ 
  ex1223 & 5.81 &   1 &   6 &  10 \\ 
  ex1223a & 5.81 &   1 &   5 &   8 \\ 
  ex1223b & 5.81 &   1 &   6 &  10 \\ 
  ex1224 & -0.88 &   1 &  12 &  19 \\ 
  ex1225 & 31.00 &   0 &   6 &   9 \\ 
  ex1226 & -17.00 &   0 &   1 &   1 \\ 
  ex1233 & 201540.00 &   3 &  21 &  49 \\ 
  ex1243 & 135552.00 &   2 &  18 &  47 \\ 
  ex1244 & 95046.40 &   1 &  13 &  17 \\ 
  ex1252 & --- & --- & --- & --- \\ 
  ex1252a & 143555.00 &   3 &  23 &  54 \\ 
  ex1263 & 69.60 &  86 & 1665 & 1773 \\ 
  ex1263a & 38.60 &  55 & 1101 & 1142 \\ 
  ex1264 & 31.60 &  11 & 185 & 222 \\ 
  ex1264a & 10.00 &   5 &  90 & 108 \\ 
  ex1265 & 22.30 &  16 & 231 & 343 \\ 
  ex1265a & 15.10 &   4 &  60 &  87 \\ 
  ex1266 & --- & --- & 62770 & 64607 \\ 
  ex1266a & 16.30 &   1 &   7 &  13 \\ 
  ex3pb & 103.58 &   2 &  19 &  37 \\ 
  ex4 & 659.57 &   4 &  46 &  77 \\ 
  fac1 & 172954000.00 &   9 &  85 & 185 \\ 
  fac2 & 407585000.00 &   5 &  35 &  97 \\ 
  fac3 & 34789500.00 &   3 &  23 &  53 \\ 
  faclay20h & 16941.00 &   0 &   1 &   1 \\ 
  faclay25 & 5107.00 &   1 &   1 &   1 \\ 
  faclay30 & 8970.00 &   3 &   1 &   1 \\ 
  faclay30h & 50775.00 &   3 &   1 &   1 \\ 
  faclay33 & 67677.00 &   6 &   1 &   1 \\ 
  faclay35 & 77040.00 &   8 &   1 &   1 \\ 
  faclay60 & 1564130.00 &  15 &   1 &   1 \\ 
  faclay70 & 1656440.00 & 2316 &   1 &   1 \\ 
  faclay75 & --- & --- &   1 &   0 \\ 
  faclay80 & --- & --- &   1 &   0 \\ 
  feedtray & -13.41 &   1 &   1 &   1 \\ 
  feedtray2 & --- & --- & --- & --- \\ 
  fin2bb & 0.00 &   2 &  19 &  26 \\ 
  flay02h & 37.95 &   2 &  20 &  40 \\ 
  flay02m & 37.95 &   2 &  24 &  45 \\ 
  flay03h & 48.99 &   4 &  26 &  60 \\ 
  flay03m & 48.99 &   3 &  29 &  61 \\ 
  flay04h & 54.99 &   6 &  50 & 100 \\ 
  flay04m & 54.99 &   4 &  42 &  85 \\ 
  flay05h & 80.99 &   7 &  36 &  97 \\ 
  flay05m & 64.50 &   5 &  41 &  96 \\ 
  flay06h & 106.88 &  11 &  38 & 121 \\ 
  flay06m & 66.93 &   6 &  43 & 102 \\ 
  fo7 & 25.60 &  87 & 1643 & 1691 \\ 
  fo7\_2 & 40.43 &  15 & 171 & 236 \\ 
  fo7\_ar2\_1 & 55.38 & 317 & 5449 & 5627 \\ 
  fo7\_ar25\_1 & --- & --- & 65860 & 66079 \\ 
  fo7\_ar3\_1 & 34.49 &  13 &  97 & 212 \\ 
  fo7\_ar4\_1 & 41.66 & 1369 & 25281 & 25479 \\ 
  fo7\_ar5\_1 & 38.09 &   7 &  53 & 110 \\ 
  fo8 & 47.40 &   7 &  66 &  99 \\ 
  fo8\_ar2\_1 & --- & --- & 62624 & 62999 \\ 
  fo8\_ar25\_1 & --- & --- & 61119 & 61451 \\ 
  fo8\_ar3\_1 & 45.34 &  38 & 469 & 590 \\ 
  fo8\_ar4\_1 & 38.66 &  36 & 384 & 529 \\ 
  fo8\_ar5\_1 & 51.84 & 370 & 5769 & 5963 \\ 
  fo9 & 52.78 &   7 &  60 & 109 \\ 
  fo9\_ar2\_1 & --- & --- & 52908 & 53324 \\ 
  fo9\_ar25\_1 & --- & --- & 57149 & 57561 \\ 
  fo9\_ar3\_1 & 49.40 &  46 & 446 & 592 \\ 
  fo9\_ar4\_1 & --- & --- & 54254 & 54687 \\ 
  fo9\_ar5\_1 & 51.17 &  11 &  72 & 131 \\ 
  fuel & 8566.12 &   2 &  16 &  42 \\ 
  fuzzy & --- & --- &   1 & 10135 \\ 
  gams01 & 26878.30 &  36 & 100 & 261 \\ 
  gams02 & 99281400.00 & 3427 & 7309 & 8005 \\ 
  gams03 & --- & --- & 10477 & 10747 \\ 
  gasnet & 6999380.00 &  10 &  97 & 194 \\ 
  gasprod\_sarawak01 & -31599.40 &   1 &   5 &   7 \\ 
  gasprod\_sarawak16 & -31399.40 &   5 &  11 &  16 \\ 
  gasprod\_sarawak81 & -31399.40 &  60 &  11 &  16 \\ 
  gastrans & --- & --- & 71678 & 71681 \\ 
  gastrans040 & 0.00 &   1 &   3 &   4 \\ 
  gastrans135 & 0.00 &  63 & 517 & 524 \\ 
  gastrans582\_cold13 & --- & --- & 24865 & 24973 \\ 
  gastrans582\_cold13\_95 & --- & --- & 26008 & 26017 \\ 
  gastrans582\_cold17 & 0.00 &  12 &  34 &  40 \\ 
  gastrans582\_cold17\_95 & 0.00 &  13 &  38 &  45 \\ 
  gastrans582\_cool12 & 0.00 &  10 &  22 &  27 \\ 
  gastrans582\_cool12\_95 & --- & --- & 25893 & 25904 \\ 
  gastrans582\_cool14 & --- & --- & 6717 & 20655 \\ 
  gastrans582\_cool14\_95 & --- & --- & 24012 & 24195 \\ 
  gastrans582\_freezing27 & --- & --- & 21961 & 22878 \\ 
  gastrans582\_freezing27\_95 & --- & --- & 649 & 19415 \\ 
  gastrans582\_freezing30 & --- & --- & 25209 & 25221 \\ 
  gastrans582\_freezing30\_95 & 0.00 & 108 & 728 & 736 \\ 
  gastrans582\_mild10 & --- & --- & 23990 & 23997 \\ 
  gastrans582\_mild10\_95 & --- & --- & 25798 & 25810 \\ 
  gastrans582\_mild11 & 0.00 &  49 & 295 & 304 \\ 
  gastrans582\_mild11\_95 & --- & --- & 25515 & 25527 \\ 
  gastrans582\_warm15 & --- & --- & 24775 & 25032 \\ 
  gastrans582\_warm15\_95 & --- & --- & 24674 & 24934 \\ 
  gastrans582\_warm31 & 0.00 &  44 & 271 & 279 \\ 
  gastrans582\_warm31\_95 & --- & --- & 26392 & 26403 \\ 
  gbd & 2.20 &   0 &   1 &   1 \\ 
  gear & 0.00 &   1 &   3 &   4 \\ 
  gear2 & 0.01 &   0 &   4 &   6 \\ 
  gear3 & 0.00 &   1 &   3 &   4 \\ 
  gear4 & 120.67 &   5 &  70 &  87 \\ 
  genpooling\_lee1 & --- & --- & --- & --- \\ 
  genpooling\_lee2 & --- & --- & --- & --- \\ 
  genpooling\_meyer04 & --- & --- & --- & --- \\ 
  genpooling\_meyer10 & --- & --- & --- & --- \\ 
  genpooling\_meyer15 & --- & --- & --- & --- \\ 
  ghg\_1veh & 7.78 &   0 &   1 &   1 \\ 
  ghg\_2veh & 7.78 &   0 &   1 &   1 \\ 
  ghg\_3veh & 7.77 &   0 &   4 &   5 \\ 
  gkocis & -1.41 &   1 &   7 &   9 \\ 
  graphpart\_2g-0044-1601 & -789955.00 &   1 &   1 &   1 \\ 
  graphpart\_2g-0055-0062 & --- & --- & 68626 & 68626 \\ 
  graphpart\_2g-0066-0066 & -2082170.00 &  11 & 110 & 215 \\ 
  graphpart\_2g-0077-0077 & --- & --- & 66436 & 66436 \\ 
  graphpart\_2g-0088-0088 & -5701630.00 &   0 &   1 &   1 \\ 
  graphpart\_2g-0099-9211 & -4495840.00 &   0 &   4 &   5 \\ 
  graphpart\_2g-1010-0824 & -6583360.00 &   0 &   1 &   1 \\ 
  graphpart\_2pm-0044-0044 & -11.00 &   0 &   1 &   1 \\ 
  graphpart\_2pm-0055-0055 & --- & --- & 70312 & 70312 \\ 
  graphpart\_2pm-0066-0066 & -27.00 &   1 &   6 &   9 \\ 
  graphpart\_2pm-0077-0777 & --- & --- & 64537 & 64537 \\ 
  graphpart\_2pm-0088-0888 & -46.00 &   0 &   4 &   5 \\ 
  graphpart\_2pm-0099-0999 & -56.00 &   1 &  18 &  22 \\ 
  graphpart\_3g-0234-0234 & --- & --- & 69513 & 69513 \\ 
  graphpart\_3g-0244-0244 & -2702200.00 &   0 &   1 &   1 \\ 
  graphpart\_3g-0333-0333 & --- & --- & 66677 & 66677 \\ 
  graphpart\_3g-0334-0334 & -3279970.00 &   0 &   1 &   1 \\ 
  graphpart\_3g-0344-0344 & --- & --- & 66770 & 66770 \\ 
  graphpart\_3g-0444-0444 & -6621100.00 &   0 &   1 &   1 \\ 
  graphpart\_3pm-0234-0234 & --- & --- & 74786 & 74786 \\ 
  graphpart\_3pm-0244-0244 & -27.00 &   0 &   1 &   1 \\ 
  graphpart\_3pm-0333-0333 & --- & --- & 71653 & 71653 \\ 
  graphpart\_3pm-0334-0334 & -33.00 &   0 &   1 &   1 \\ 
  graphpart\_3pm-0344-0344 & --- & --- & 68708 & 68708 \\ 
  graphpart\_3pm-0444-0444 & -57.00 &   1 &   4 &   5 \\ 
  graphpart\_clique-20 & 147.00 &   0 &   1 &   1 \\ 
  graphpart\_clique-30 & 495.00 &   0 &   1 &   1 \\ 
  graphpart\_clique-40 & 1183.00 &   0 &   1 &   1 \\ 
  graphpart\_clique-50 & --- & --- & 60889 & 60889 \\ 
  graphpart\_clique-60 & 4010.00 &   2 &  34 &  36 \\ 
  graphpart\_clique-70 & --- & --- & 50784 & 50784 \\ 
  hda & -4322.55 &   6 &  18 &  39 \\ 
  heatexch\_gen1 & 410804.00 &   2 &  19 &  42 \\ 
  heatexch\_gen2 & 739019.00 &   1 &  20 &  26 \\ 
  heatexch\_gen3 & 109097.00 &   4 &  28 &  46 \\ 
  heatexch\_spec1 & 219858.00 &   2 &  18 &  42 \\ 
  heatexch\_spec2 & 849922.00 &   0 &   3 &   4 \\ 
  heatexch\_spec3 & 319465.00 &   1 &   8 &  13 \\ 
  heatexch\_trigen & 977262.00 & 144 & 2367 & 2602 \\ 
  hmittelman & 16.00 &   2 &  22 &  39 \\ 
  hybriddynamic\_fixed & 1.47 &   0 &   7 &  12 \\ 
  hybriddynamic\_var & 1.54 &   0 &   5 &   9 \\ 
  hydroenergy1 & 207178.00 &   2 &  29 &  39 \\ 
  hydroenergy2 & 369251.00 &   4 &  31 &  42 \\ 
  hydroenergy3 & 742404.00 &   7 &  31 &  44 \\ 
  ibs2 & 4.88 & 1032 &  55 & 125 \\ 
  jit1 & 173983.00 &   0 &   3 &   4 \\ 
  johnall & -222.37 &   5 &  54 &  57 \\ 
  kport20 & 33.50 &   2 &  14 &  37 \\ 
  kport40 & 42.02 &  70 & 1184 & 1289 \\ 
  lip & 5428650.00 &   0 &   3 &   6 \\ 
  lop97ic & 4535.18 &   7 &  10 &  31 \\ 
  lop97icx & 4590.48 &   3 &  13 &  31 \\ 
  m3 & 55.80 &   2 &  21 &  32 \\ 
  m6 & 123.98 &  13 & 105 & 206 \\ 
  m7 & --- & --- & 66377 & 66664 \\ 
  m7\_ar2\_1 & --- & --- & 65520 & 65796 \\ 
  m7\_ar25\_1 & --- & --- & 69125 & 69424 \\ 
  m7\_ar3\_1 & --- & --- & 66925 & 67206 \\ 
  m7\_ar4\_1 & 450.97 & 127 & 1993 & 2166 \\ 
  m7\_ar5\_1 & 511.21 &  61 & 857 & 1025 \\ 
  mbtd & 5.58 &  66 &  13 &  41 \\ 
  meanvarx & 14.37 &   1 &  23 &  26 \\ 
  meanvarxsc & --- & --- & --- & --- \\ 
  milinfract & 2.63 &   8 &   5 &   7 \\ 
  minlphix & 345.51 &   3 &  53 &  63 \\ 
  multiplants\_mtg1a & 363.57 &   2 &  17 &  35 \\ 
  multiplants\_mtg1b & 212.21 &  33 & 602 & 617 \\ 
  multiplants\_mtg1c & 406.98 &   5 &  31 &  52 \\ 
  multiplants\_mtg2 & 7051.41 &   4 &  65 &  75 \\ 
  multiplants\_mtg5 & 4706.88 &  75 & 1228 & 1258 \\ 
  multiplants\_mtg6 & 4032.27 &  13 & 155 & 177 \\ 
  multiplants\_stg1 & 250.44 & 1000 & 16363 & 16375 \\ 
  multiplants\_stg1a & 38.85 &  19 & 267 & 284 \\ 
  multiplants\_stg1b & 136.38 &  35 & 498 & 506 \\ 
  multiplants\_stg1c & --- & --- & 61799 & 61808 \\ 
  multiplants\_stg5 & --- & --- & 52743 & 52952 \\ 
  multiplants\_stg6 & --- & --- & 45323 & 46702 \\ 
  ndcc12 & 108.11 &   8 &  59 & 106 \\ 
  ndcc12persp & --- & --- & --- & --- \\ 
  ndcc13 & 94.17 &   6 &  17 &  56 \\ 
  ndcc13persp & --- & --- & --- & --- \\ 
  ndcc14 & 130.16 &  19 &  70 & 247 \\ 
  ndcc14persp & --- & --- & --- & --- \\ 
  ndcc15 & 95.28 &   4 &  22 &  50 \\ 
  ndcc15persp & --- & --- & --- & --- \\ 
  ndcc16 & 131.63 &   7 &  28 &  67 \\ 
  ndcc16persp & --- & --- & --- & --- \\ 
  netmod\_dol1 & -0.01 &  18 &  22 &  40 \\ 
  netmod\_dol2 & -0.49 &  37 &  16 &  44 \\ 
  netmod\_kar1 & 0.00 &   3 &  17 &  31 \\ 
  netmod\_kar2 & 0.00 &   3 &  17 &  31 \\ 
  no7\_ar2\_1 & --- & --- & 66159 & 66423 \\ 
  no7\_ar25\_1 & 175.87 & 156 & 2635 & 2871 \\ 
  no7\_ar3\_1 & 149.22 &  13 & 142 & 201 \\ 
  no7\_ar4\_1 & 188.94 & 141 & 2040 & 2246 \\ 
  no7\_ar5\_1 & 153.84 & 148 & 2440 & 2576 \\ 
  nous1 & 1.57 &   0 &   3 &   4 \\ 
  nous2 & 0.63 &   0 &   3 &   4 \\ 
  nuclear104 & --- & --- & --- & --- \\ 
  nuclear10a & --- & --- & 1737 & 2056 \\ 
  nuclear10b & -1.15 & 1604 &  31 & 124 \\ 
  nuclear14 & -1.13 &   3 &   4 &   5 \\ 
  nuclear14a & -1.13 &   1 &   1 &   1 \\ 
  nuclear14b & -1.09 &  13 &  51 &  73 \\ 
  nuclear25 & -1.12 &   7 &  22 &  29 \\ 
  nuclear25a & -1.12 &   5 &  22 &  29 \\ 
  nuclear25b & -1.09 &  29 &  63 & 153 \\ 
  nuclear49 & -1.15 &  30 &  13 &  17 \\ 
  nuclear49a & -1.15 &  26 &  25 &  33 \\ 
  nuclear49b & -1.13 & 133 &  25 &  74 \\ 
  nuclearva & -1.01 &   3 &  33 &  45 \\ 
  nuclearvb & -1.02 &   1 &   1 &   1 \\ 
  nuclearvc & -0.98 &   5 &  49 &  59 \\ 
  nuclearvd & -1.04 &  10 &  63 &  75 \\ 
  nuclearve & -1.02 &   1 &   7 &   9 \\ 
  nuclearvf & --- & --- & --- & --- \\ 
  nvs01 & 13.10 &   1 &   5 &   7 \\ 
  nvs02 & 5.96 &   2 &  35 &  38 \\ 
  nvs03 & 16.00 &   0 &  11 &  14 \\ 
  nvs04 & --- & --- & --- & --- \\ 
  nvs05 & 5.47 &   0 &   3 &   4 \\ 
  nvs06 & 1.86 &   1 &   3 &   4 \\ 
  nvs07 & 4.00 &   1 &   5 &   6 \\ 
  nvs08 & 23.83 &   1 &  10 &  14 \\ 
  nvs09 & -43.13 &   0 &   1 &   1 \\ 
  nvs10 & -308.40 &   0 &   4 &   7 \\ 
  nvs11 & -416.40 &   3 &  69 &  74 \\ 
  nvs12 & -477.00 &   0 &   3 &   5 \\ 
  nvs13 & -585.20 &   1 &   3 &   5 \\ 
  nvs14 & -40358.20 &   2 &  35 &  38 \\ 
  nvs15 & 1.00 &   0 &   3 &   4 \\ 
  nvs16 & 14.20 &   7 &  55 & 108 \\ 
  nvs17 & -1078.20 &   3 &  52 &  57 \\ 
  nvs18 & -778.40 &   0 &   3 &   5 \\ 
  nvs19 & -1070.00 &   1 &  10 &  17 \\ 
  nvs20 & 230.92 &   1 &   7 &  13 \\ 
  nvs21 & -4.27 &   8 & 172 & 174 \\ 
  nvs22 & 6.06 &   0 &   4 &   6 \\ 
  nvs23 & -1078.40 &   1 &   3 &   5 \\ 
  nvs24 & -1001.00 &   1 &   7 &  14 \\ 
  o7 & 190.62 &  25 & 374 & 418 \\ 
  o7\_2 & 161.38 &  17 & 147 & 281 \\ 
  o7\_ar2\_1 & --- & --- & 61789 & 62056 \\ 
  o7\_ar25\_1 & 160.47 & 173 & 2936 & 3100 \\ 
  o7\_ar3\_1 & 162.83 &  34 & 446 & 562 \\ 
  o7\_ar4\_1 & 182.29 &  69 & 971 & 1138 \\ 
  o7\_ar5\_1 & 166.99 &  85 & 1274 & 1443 \\ 
  o8\_ar4\_1 & 345.47 &  22 & 130 & 256 \\ 
  o9\_ar4\_1 & 341.57 &  55 & 497 & 691 \\ 
  oaer & -1.92 &   1 &   4 &   5 \\ 
  oil & -0.87 &   2 &   3 &   4 \\ 
  oil2 & -0.73 &   1 &   3 &   4 \\ 
  ortez & -9532.04 &   2 &  13 &  20 \\ 
  parallel & --- & --- & 65520 & 65520 \\ 
  pb302035 & 4052260.00 &  28 &  25 &  34 \\ 
  pb302055 & 4087020.00 & 152 & 155 & 200 \\ 
  pb302075 & 4624370.00 &  62 &  49 &  77 \\ 
  pb302095 & --- & --- & 3207 & 4701 \\ 
  pb351535 & 5474850.00 & 180 & 240 & 306 \\ 
  pb351555 & 5256790.00 &  26 &  21 &  41 \\ 
  pb351575 & 6457260.00 &  16 &  15 &  21 \\ 
  pb351595 & 11847500.00 & 915 & 919 & 1524 \\ 
  pooling\_epa1 & -280.81 &   8 & 110 & 134 \\ 
  pooling\_epa2 & -4268.72 &   3 &  36 &  44 \\ 
  pooling\_epa3 & --- & --- & --- & --- \\ 
  portfol\_buyin & 0.04 &   1 &   3 &   4 \\ 
  portfol\_card & 0.03 &   1 &   5 &   8 \\ 
  portfol\_classical050\_1 & -0.09 &   1 &  10 &  21 \\ 
  portfol\_classical200\_2 & -0.06 &   8 &  18 &  50 \\ 
  portfol\_robust050\_34 & -0.00 &   1 &  11 &  16 \\ 
  portfol\_robust100\_09 & -0.08 &   2 &  11 &  23 \\ 
  portfol\_robust200\_03 & -0.08 &  12 &  16 &  44 \\ 
  portfol\_roundlot & 0.15 & 3412 & 5488 & 69538 \\ 
  portfol\_shortfall050\_68 & -1.06 &   1 &  10 &  15 \\ 
  portfol\_shortfall100\_04 & -1.07 &   3 &  13 &  25 \\ 
  portfol\_shortfall200\_05 & -1.06 &   7 &  14 &  27 \\ 
  primary & --- & --- & --- & --- \\ 
  prob02 & 112235.00 &   1 &  10 &  13 \\ 
  prob03 & 11.00 &   2 &  18 &  31 \\ 
  prob10 & 3.45 &   0 &   3 &   4 \\ 
  procsel & -1.41 &   0 &   5 &   7 \\ 
  product & -2075.33 &  16 &  56 & 103 \\ 
  product2 & --- & --- & --- & --- \\ 
  qap & 411560.00 &   6 &  13 &  28 \\ 
  qapw & 395664.00 &  48 & 106 & 345 \\ 
  ravempb & 269590.00 &   1 &  10 &  23 \\ 
  risk2bpb & -55.48 &   1 &  17 &  21 \\ 
  routingdelay\_bigm & --- & --- & 27714 & 27777 \\ 
  routingdelay\_proj & 149.70 &  25 & 145 & 170 \\ 
  rsyn0805h & 1280.01 &   1 &  17 &  26 \\ 
  rsyn0805m & 928.36 &   3 &  24 &  50 \\ 
  rsyn0805m02h & 2121.61 &   5 &   9 &  24 \\ 
  rsyn0805m02m & 923.55 &   6 &  34 &  80 \\ 
  rsyn0805m03h & 2996.90 &   7 &   8 &  19 \\ 
  rsyn0805m03m & 2256.38 &   8 &  33 &  77 \\ 
  rsyn0805m04h & 7145.11 &   7 &   5 &  11 \\ 
  rsyn0805m04m & 4641.20 &  10 &  32 &  78 \\ 
  rsyn0810h & 1675.12 &   2 &  18 &  26 \\ 
  rsyn0810m & 801.54 &   3 &  21 &  49 \\ 
  rsyn0810m02h & 1631.39 &   5 &   9 &  24 \\ 
  rsyn0810m02m & 987.64 &   6 &  32 &  76 \\ 
  rsyn0810m03h & 2671.49 &   8 &   7 &  18 \\ 
  rsyn0810m03m & 2355.46 &  20 & 162 & 226 \\ 
  rsyn0810m04h & 6472.92 &  23 &  12 &  26 \\ 
  rsyn0810m04m & 4525.62 &  11 &  25 &  71 \\ 
  rsyn0815h & 1262.05 &   1 &  11 &  21 \\ 
  rsyn0815m & 923.47 &   3 &  26 &  52 \\ 
  rsyn0815m02h & 1579.85 &   6 &  11 &  30 \\ 
  rsyn0815m02m & 598.63 &   6 &  32 &  71 \\ 
  rsyn0815m03h & 2702.87 &  10 &   8 &  20 \\ 
  rsyn0815m03m & 2697.07 &  10 &  35 &  81 \\ 
  rsyn0815m04h & 3220.26 &  21 &   9 &  24 \\ 
  rsyn0815m04m & 1329.02 &  13 &  30 &  79 \\ 
  rsyn0820h & 1138.41 &   2 &   9 &  17 \\ 
  rsyn0820m & 723.84 &   4 &  24 &  56 \\ 
  rsyn0820m02h & 1005.44 &   6 &   8 &  18 \\ 
  rsyn0820m02m & 258.97 &   7 &  32 &  84 \\ 
  rsyn0820m03h & 1984.43 &  14 &   8 &  16 \\ 
  rsyn0820m03m & 1800.99 &  17 &  52 & 113 \\ 
  rsyn0820m04h & 2363.62 &  25 &   8 &  23 \\ 
  rsyn0820m04m & 1642.15 &  18 &  30 &  89 \\ 
  rsyn0830h & 506.70 &   2 &  10 &  19 \\ 
  rsyn0830m & 203.31 &   4 &  30 &  64 \\ 
  rsyn0830m02h & 661.04 &   4 &   8 &  18 \\ 
  rsyn0830m02m & -273.98 &   6 &  33 &  68 \\ 
  rsyn0830m03h & 1431.22 &  10 &   8 &  18 \\ 
  rsyn0830m03m & 137.73 &  12 &  34 &  85 \\ 
  rsyn0830m04h & 2284.01 &  21 &   9 &  19 \\ 
  rsyn0830m04m & 128.87 &  19 &  36 &  90 \\ 
  rsyn0840h & 297.64 &   2 &   6 &  14 \\ 
  rsyn0840m & 54.73 &   5 &  30 &  69 \\ 
  rsyn0840m02h & 612.11 &   5 &  10 &  18 \\ 
  rsyn0840m02m & -362.91 &   7 &  29 &  70 \\ 
  rsyn0840m03h & 2664.97 &  15 &  11 &  20 \\ 
  rsyn0840m03m & 2121.79 &  12 &  30 &  72 \\ 
  rsyn0840m04h & 2327.32 &  21 &   8 &  16 \\ 
  rsyn0840m04m & -18.45 &  21 &  32 &  84 \\ 
  saa\_2 & 12.75 &  35 &  50 &  78 \\ 
  sep1 & -510.08 &   0 &   5 &   7 \\ 
  sepasequ\_complex & 538.16 &  12 &  55 &  97 \\ 
  sepasequ\_convent & 514.98 &  11 &  24 &  34 \\ 
  sfacloc1\_2\_80 & 13.55 &   1 &   3 &   4 \\ 
  sfacloc1\_2\_90 & 29.57 &   0 &   3 &   4 \\ 
  sfacloc1\_2\_95 & 18.85 &   0 &   1 &   1 \\ 
  sfacloc1\_3\_80 & 9.10 &   0 &   3 &   4 \\ 
  sfacloc1\_3\_90 & 11.88 &   0 &   3 &   4 \\ 
  sfacloc1\_3\_95 & 12.46 &   0 &   1 &   1 \\ 
  sfacloc1\_4\_80 & 8.46 &   1 &   3 &   4 \\ 
  sfacloc1\_4\_90 & 12.00 &   0 &   3 &   4 \\ 
  sfacloc1\_4\_95 & 11.24 &   1 &   3 &   4 \\ 
  sfacloc2\_2\_80 & --- & --- & 38561 & 38572 \\ 
  sfacloc2\_2\_90 & 29.96 &   7 &  77 & 109 \\ 
  sfacloc2\_2\_95 & 30.30 &  92 & 1745 & 1788 \\ 
  sfacloc2\_3\_80 & 25.20 &  15 &  82 & 127 \\ 
  sfacloc2\_3\_90 & 29.73 &   9 & 103 & 143 \\ 
  sfacloc2\_3\_95 & 25.17 &   3 &  27 &  54 \\ 
  sfacloc2\_4\_80 & 21.03 & 156 & 1403 & 1510 \\ 
  sfacloc2\_4\_90 & 28.72 &  87 & 1272 & 1390 \\ 
  sfacloc2\_4\_95 & 26.74 & 682 & 11503 & 11784 \\ 
  slay04h & 14763.90 &   3 &  20 &  44 \\ 
  slay04m & 12200.30 &   1 &  20 &  37 \\ 
  slay05h & 24164.50 &   3 &  16 &  38 \\ 
  slay05m & 23290.00 &   1 &  14 &  28 \\ 
  slay06h & 38850.80 &   3 &  19 &  46 \\ 
  slay06m & 33168.50 &   2 &  16 &  35 \\ 
  slay07h & 69708.50 &   5 &  26 &  58 \\ 
  slay07m & 70870.80 &   3 &  27 &  60 \\ 
  slay08h & 105595.00 &   7 &  28 &  71 \\ 
  slay08m & 85921.30 &   3 &  25 &  52 \\ 
  slay09h & 156117.00 &  10 &  30 &  83 \\ 
  slay09m & 160234.00 &   5 &  31 &  86 \\ 
  slay10h & 197463.00 &  18 &  37 & 106 \\ 
  slay10m & 176821.00 &   4 &  29 &  73 \\ 
  smallinvDAXr1b010-011 & --- & --- & --- & --- \\ 
  smallinvDAXr1b020-022 & --- & --- & --- & --- \\ 
  smallinvDAXr1b050-055 & --- & --- & --- & --- \\ 
  smallinvDAXr1b100-110 & --- & --- & --- & --- \\ 
  smallinvDAXr1b150-165 & --- & --- & --- & --- \\ 
  smallinvDAXr1b200-220 & --- & --- & --- & --- \\ 
  smallinvDAXr2b010-011 & --- & --- & --- & --- \\ 
  smallinvDAXr2b020-022 & --- & --- & --- & --- \\ 
  smallinvDAXr2b050-055 & --- & --- & --- & --- \\ 
  smallinvDAXr2b100-110 & --- & --- & --- & --- \\ 
  smallinvDAXr2b150-165 & --- & --- & --- & --- \\ 
  smallinvDAXr2b200-220 & --- & --- & --- & --- \\ 
  smallinvDAXr3b010-011 & --- & --- & --- & --- \\ 
  smallinvDAXr3b020-022 & --- & --- & --- & --- \\ 
  smallinvDAXr3b050-055 & --- & --- & --- & --- \\ 
  smallinvDAXr3b100-110 & --- & --- & --- & --- \\ 
  smallinvDAXr3b150-165 & --- & --- & --- & --- \\ 
  smallinvDAXr3b200-220 & --- & --- & --- & --- \\ 
  smallinvDAXr4b010-011 & --- & --- & --- & --- \\ 
  smallinvDAXr4b020-022 & --- & --- & --- & --- \\ 
  smallinvDAXr4b050-055 & --- & --- & --- & --- \\ 
  smallinvDAXr4b100-110 & --- & --- & --- & --- \\ 
  smallinvDAXr4b150-165 & --- & --- & --- & --- \\ 
  smallinvDAXr4b200-220 & --- & --- & --- & --- \\ 
  smallinvDAXr5b010-011 & --- & --- & --- & --- \\ 
  smallinvDAXr5b020-022 & --- & --- & --- & --- \\ 
  smallinvDAXr5b050-055 & --- & --- & --- & --- \\ 
  smallinvDAXr5b100-110 & --- & --- & --- & --- \\ 
  smallinvDAXr5b150-165 & --- & --- & --- & --- \\ 
  smallinvDAXr5b200-220 & --- & --- & --- & --- \\ 
  smallinvSNPr1b010-011 & --- & --- & --- & --- \\ 
  smallinvSNPr1b020-022 & --- & --- & --- & --- \\ 
  smallinvSNPr1b050-055 & --- & --- & --- & --- \\ 
  smallinvSNPr1b100-110 & --- & --- & --- & --- \\ 
  smallinvSNPr1b150-165 & --- & --- & --- & --- \\ 
  smallinvSNPr1b200-220 & --- & --- & --- & --- \\ 
  smallinvSNPr2b010-011 & --- & --- & --- & --- \\ 
  smallinvSNPr2b020-022 & --- & --- & --- & --- \\ 
  smallinvSNPr2b050-055 & --- & --- & --- & --- \\ 
  smallinvSNPr2b100-110 & --- & --- & --- & --- \\ 
  smallinvSNPr2b150-165 & --- & --- & --- & --- \\ 
  smallinvSNPr2b200-220 & --- & --- & --- & --- \\ 
  smallinvSNPr3b010-011 & --- & --- & --- & --- \\ 
  smallinvSNPr3b020-022 & --- & --- & --- & --- \\ 
  smallinvSNPr3b050-055 & --- & --- & --- & --- \\ 
  smallinvSNPr3b100-110 & --- & --- & --- & --- \\ 
  smallinvSNPr3b150-165 & --- & --- & --- & --- \\ 
  smallinvSNPr3b200-220 & --- & --- & --- & --- \\ 
  smallinvSNPr4b010-011 & --- & --- & --- & --- \\ 
  smallinvSNPr4b020-022 & --- & --- & --- & --- \\ 
  smallinvSNPr4b050-055 & --- & --- & --- & --- \\ 
  smallinvSNPr4b100-110 & --- & --- & --- & --- \\ 
  smallinvSNPr4b150-165 & --- & --- & --- & --- \\ 
  smallinvSNPr4b200-220 & --- & --- & --- & --- \\ 
  smallinvSNPr5b010-011 & --- & --- & --- & --- \\ 
  smallinvSNPr5b020-022 & --- & --- & --- & --- \\ 
  smallinvSNPr5b050-055 & --- & --- & --- & --- \\ 
  smallinvSNPr5b100-110 & --- & --- & --- & --- \\ 
  smallinvSNPr5b150-165 & --- & --- & --- & --- \\ 
  smallinvSNPr5b200-220 & --- & --- & --- & --- \\ 
  space25 & 919.99 & 279 & 3235 & 4234 \\ 
  space25a & 657.73 &   5 &  80 & 106 \\ 
  space960 & 7985000.00 & 106 &  39 & 111 \\ 
  spectra2 & 14.04 &   3 &  25 &  53 \\ 
  sporttournament06 & 10.00 &   0 &   3 &   4 \\ 
  sporttournament08 & 22.00 &   0 &   1 &   1 \\ 
  sporttournament10 & 38.00 &   0 &   1 &   1 \\ 
  sporttournament12 & 58.00 &   0 &   1 &   1 \\ 
  sporttournament14 & 80.00 &   0 &   1 &   1 \\ 
  sporttournament16 & 100.00 &   0 &   3 &   4 \\ 
  sporttournament18 & 134.00 &   0 &   1 &   1 \\ 
  sporttournament20 & 162.00 &   1 &   3 &   4 \\ 
  sporttournament22 & 198.00 &   0 &   3 &   4 \\ 
  sporttournament24 & 230.00 &   0 &   1 &   1 \\ 
  sporttournament26 & 282.00 &   0 &   1 &   1 \\ 
  sporttournament28 & 298.00 &   0 &   3 &   4 \\ 
  sporttournament30 & 318.00 &   0 &   1 &   1 \\ 
  sporttournament32 & 384.00 &   0 &   1 &   1 \\ 
  sporttournament34 & 438.00 &   1 &   1 &   1 \\ 
  sporttournament36 & 472.00 &   0 &   1 &   1 \\ 
  sporttournament38 & 526.00 &   0 &   1 &   1 \\ 
  sporttournament40 & 590.00 &   0 &   1 &   1 \\ 
  sporttournament42 & 666.00 &   0 &   1 &   1 \\ 
  sporttournament44 & 706.00 &   0 &   1 &   1 \\ 
  sporttournament46 & 806.00 &   0 &   1 &   1 \\ 
  sporttournament48 & 904.00 &   0 &   1 &   1 \\ 
  sporttournament50 & 948.00 &   1 &   1 &   1 \\ 
  spring & 1.35 &   1 &  19 &  27 \\ 
  squfl010-025 & 233.69 &   3 &  20 &  41 \\ 
  squfl010-025persp & 214.11 & 247 &   3 & 776 \\ 
  squfl010-040 & 283.98 &   5 &  21 &  43 \\ 
  squfl010-040persp & 240.60 &  72 &   1 & 100 \\ 
  squfl010-080 & 523.50 &   8 &  22 &  45 \\ 
  squfl010-080persp & 509.70 & 1418 &   4 & 1684 \\ 
  squfl015-060 & 422.51 &   5 &  24 &  49 \\ 
  squfl015-060persp & 366.62 & 522 &   8 & 367 \\ 
  squfl015-080 & 596.93 &  10 &  29 &  56 \\ 
  squfl015-080persp & 407.13 & 994 &   3 & 608 \\ 
  squfl020-040 & 286.55 &   4 &  23 &  45 \\ 
  squfl020-040persp & 221.02 & 339 &  11 & 249 \\ 
  squfl020-050 & 323.25 &   7 &  25 &  50 \\ 
  squfl020-050persp & 244.05 & 331 &  11 & 193 \\ 
  squfl020-150 & 1131.79 & 113 &  36 &  73 \\ 
  squfl020-150persp & 562.07 & 2348 &  15 & 384 \\ 
  squfl025-025 & 244.44 &   4 &  22 &  43 \\ 
  squfl025-025persp & 178.63 & 719 &  10 & 645 \\ 
  squfl025-030 & 215.34 &   3 &  19 &  37 \\ 
  squfl025-030persp & 215.34 & 1367 &  10 & 1057 \\ 
  squfl025-040 & 299.01 &   7 &  24 &  47 \\ 
  squfl025-040persp & 225.39 & 650 &  29 & 385 \\ 
  squfl030-100 & 766.47 &  36 &  33 &  66 \\ 
  squfl030-100persp & --- & --- &   3 & 535 \\ 
  squfl030-150 & 628.11 & 374 &  31 &  61 \\ 
  squfl030-150persp & --- & --- &  10 & 307 \\ 
  squfl040-080 & 294.95 &  50 &  22 &  43 \\ 
  squfl040-080persp & 578.72 & 2651 &  23 & 378 \\ 
  sssd08-04 & 203796.00 &   2 &  12 &  36 \\ 
  sssd08-04persp & 413575.00 &   2 &  18 &  38 \\ 
  sssd12-05 & 421690.00 &   3 &  18 &  51 \\ 
  sssd12-05persp & 458644.00 &   2 &  21 &  52 \\ 
  sssd15-04 & 266257.00 &   1 &  14 &  37 \\ 
  sssd15-04persp & 281947.00 &   3 &  21 &  47 \\ 
  sssd15-06 & 782066.00 &   4 &  21 &  61 \\ 
  sssd15-06persp & 553100.00 &   2 &  12 &  28 \\ 
  sssd15-08 & 882079.00 &   5 &  26 &  77 \\ 
  sssd15-08persp & 599059.00 &   3 &  17 &  49 \\ 
  sssd16-07 & 666186.00 &   4 &  25 &  78 \\ 
  sssd16-07persp & 421588.00 &   2 &  14 &  29 \\ 
  sssd18-06 & 438633.00 &   3 &  14 &  41 \\ 
  sssd18-06persp & 484779.00 &   2 &  18 &  47 \\ 
  sssd18-08 & 949385.00 &   5 &  46 & 100 \\ 
  sssd18-08persp & 1054060.00 &   4 &  21 &  67 \\ 
  sssd20-04 & 402023.00 &   7 &  16 &  42 \\ 
  sssd20-04persp & 403765.00 &   3 &  20 &  58 \\ 
  sssd20-08 & 522836.00 &   3 &  14 &  59 \\ 
  sssd20-08persp & 568422.00 &   3 &  21 &  55 \\ 
  sssd22-08 & 600312.00 &   5 &  22 &  59 \\ 
  sssd22-08persp & 567598.00 &   3 &  22 &  59 \\ 
  sssd25-04 & 302373.00 &   2 &  19 &  46 \\ 
  sssd25-04persp & 300946.00 &   2 &  30 &  55 \\ 
  sssd25-08 & 585480.00 &   3 &  19 &  55 \\ 
  sssd25-08persp & 601061.00 &   3 &  20 &  51 \\ 
  st\_e13 & --- & --- & --- & --- \\ 
  st\_e14 & 5.81 &   0 &   6 &  10 \\ 
  st\_e15 & --- & --- & --- & --- \\ 
  st\_e27 & 2.00 &   0 &   1 &   1 \\ 
  st\_e29 & -0.88 &   1 &  12 &  19 \\ 
  st\_e31 & --- & --- & --- & --- \\ 
  st\_e32 & -0.72 &  14 & 243 & 289 \\ 
  st\_e35 & 71468.10 &   1 &   3 &   4 \\ 
  st\_e36 & --- & --- & 8225 & 51928 \\ 
  st\_e38 & 7197.73 &   0 &   1 &   1 \\ 
  st\_e40 & --- & --- & --- & --- \\ 
  st\_miqp1 & 380.50 &   1 &  15 &  28 \\ 
  st\_miqp2 & 2.00 &   1 &  12 &  16 \\ 
  st\_miqp3 & --- & --- & --- & --- \\ 
  st\_miqp4 & --- & --- & --- & --- \\ 
  st\_miqp5 & -333.89 &   0 &   1 &   1 \\ 
  stockcycle & 260786.00 &  12 & 147 & 218 \\ 
  st\_test1 & 0.00 &   1 &   6 &  11 \\ 
  st\_test2 & --- & --- & --- & --- \\ 
  st\_test3 & --- & --- & --- & --- \\ 
  st\_test4 & --- & --- & --- & --- \\ 
  st\_test5 & -110.00 &   2 &  28 &  46 \\ 
  st\_test6 & 664.00 &   3 &  46 &  75 \\ 
  st\_test8 & -29605.00 &   0 &   4 &   5 \\ 
  st\_testgr1 & -12.76 &   1 &  11 &  15 \\ 
  st\_testgr3 & -20.50 &   2 &  19 &  29 \\ 
  st\_testph4 & -80.50 &   0 &   3 &   5 \\ 
  super1 & --- & --- & --- & --- \\ 
  super2 & --- & --- & --- & --- \\ 
  super3 & --- & --- & --- & --- \\ 
  super3t & --- & --- & --- & --- \\ 
  supplychain & --- & --- & --- & --- \\ 
  supplychainp1\_020306 & 593975.00 &   0 &   3 &   5 \\ 
  supplychainp1\_022020 & 3378850.00 &   5 &   5 &   9 \\ 
  supplychainp1\_030510 & 1215690.00 &   0 &   3 &   5 \\ 
  supplychainp1\_053050 & 8543940.00 & 111 &   5 &  11 \\ 
  supplychainr1\_020306 & 628623.00 &   0 &   5 &   9 \\ 
  supplychainr1\_022020 & 4050140.00 &   8 &  14 &  43 \\ 
  supplychainr1\_030510 & 1568450.00 &   1 &  10 &  16 \\ 
  supplychainr1\_053050 & 8588080.00 & 180 &  32 & 122 \\ 
  syn05h & 837.73 &   0 &   3 &   4 \\ 
  syn05m & 831.45 &   1 &  10 &  18 \\ 
  syn05m02h & 3032.74 &   0 &   3 &   4 \\ 
  syn05m02m & 3026.74 &   1 &  13 &  25 \\ 
  syn05m03h & 4027.37 &   0 &   3 &   4 \\ 
  syn05m03m & 3987.03 &   1 &  10 &  22 \\ 
  syn05m04h & 5510.39 &   0 &   3 &   4 \\ 
  syn05m04m & 5499.39 &   1 &  13 &  28 \\ 
  syn10h & 1267.35 &   0 &   3 &   4 \\ 
  syn10m & 560.25 &   0 &  10 &  14 \\ 
  syn10m02h & 2310.30 &   1 &   3 &   4 \\ 
  syn10m02m & 2275.85 &   2 &  13 &  25 \\ 
  syn10m03h & 3354.68 &   1 &   3 &   4 \\ 
  syn10m03m & 2773.64 &   2 &  25 &  38 \\ 
  syn10m04h & 4557.06 &   0 &   3 &   4 \\ 
  syn10m04m & 3814.77 &   3 &  26 &  41 \\ 
  syn15h & 853.28 &   0 &   3 &   4 \\ 
  syn15m & 503.63 &   1 &  10 &  25 \\ 
  syn15m02h & 2832.75 &   0 &   3 &   4 \\ 
  syn15m02m & 2777.75 &   1 &  14 &  25 \\ 
  syn15m03h & 3850.18 &   0 &   3 &   4 \\ 
  syn15m03m & 3795.18 &   3 &  17 &  34 \\ 
  syn15m04h & 4937.48 &   1 &   3 &   4 \\ 
  syn15m04m & 4882.48 &   2 &  18 &  31 \\ 
  syn20h & 924.26 &   1 &   3 &   4 \\ 
  syn20m & 722.64 &   2 &   8 &  19 \\ 
  syn20m02h & 1752.13 &   1 &   3 &   4 \\ 
  syn20m02m & 1693.13 &   2 &  23 &  38 \\ 
  syn20m03h & 2646.95 &   1 &   3 &   4 \\ 
  syn20m03m & 2574.01 &   3 &  18 &  41 \\ 
  syn20m04h & 3532.74 &   1 &   3 &   4 \\ 
  syn20m04m & 3475.74 &   3 &  18 &  35 \\ 
  syn30h & 134.03 &   0 &   3 &   5 \\ 
  syn30m & -38.75 &   2 &  15 &  28 \\ 
  syn30m02h & 393.25 &   1 &   3 &   5 \\ 
  syn30m02m & -13.34 &   2 &  16 &  29 \\ 
  syn30m03h & 646.05 &   2 &   3 &   5 \\ 
  syn30m03m & 31.70 &   3 &  15 &  34 \\ 
  syn30m04h & 859.05 &   3 &   3 &   5 \\ 
  syn30m04m & 111.41 &   5 &  13 &  33 \\ 
  syn40h & 58.66 &   1 &   3 &   5 \\ 
  syn40m & -25.08 &   2 &  27 &  46 \\ 
  syn40m02h & 379.76 &   1 &   3 &   4 \\ 
  syn40m02m & 149.17 &   3 &  11 &  30 \\ 
  syn40m03h & 390.15 &   4 &   4 &   6 \\ 
  syn40m03m & -12.17 &   6 &  28 &  51 \\ 
  syn40m04h & 896.96 &   4 &   3 &   4 \\ 
  syn40m04m & 609.95 &   6 &  12 &  32 \\ 
  synheat & 219858.00 &   2 &  18 &  43 \\ 
  synthes1 & 7.09 &   1 &   8 &  14 \\ 
  synthes2 & 80.29 &   2 &  17 &  34 \\ 
  synthes3 & 82.37 &   1 &  14 &  28 \\ 
  tanksize & 1.27 &   0 &   6 &  10 \\ 
  telecomsp\_metro & --- & --- &   1 &   1 \\ 
  telecomsp\_njlata & --- & --- &   1 &   2 \\ 
  telecomsp\_nor\_sun & --- & --- &   1 &   0 \\ 
  telecomsp\_pacbell & 336210.00 & 3146 &  35 & 176 \\ 
  tln12 & 315.60 & 207 & 2860 & 3350 \\ 
  tln2 & 23.30 &  15 & 284 & 313 \\ 
  tln4 & 9.30 &   5 &  92 & 103 \\ 
  tln5 & 16.50 &   3 &  47 &  65 \\ 
  tln6 & 20.50 &  13 & 177 & 209 \\ 
  tln7 & 39.00 & 157 & 3173 & 3239 \\ 
  tloss & 24.30 &   2 &  41 &  62 \\ 
  tls12 & --- & --- & 37596 & 45964 \\ 
  tls2 & 11.30 & 193 & 3385 & 3543 \\ 
  tls4 & 22.00 & 492 & 9538 & 9781 \\ 
  tls5 & 15.50 &  81 & 1319 & 1590 \\ 
  tls6 & 36.10 & 152 & 2418 & 2859 \\ 
  tls7 & 48.80 & 1693 & 27458 & 28995 \\ 
  tltr & 54.60 &   2 &  23 &  32 \\ 
  transswitch0009p & --- & --- & --- & --- \\ 
  transswitch0009r & --- & --- & --- & --- \\ 
  transswitch0014p & --- & --- & --- & --- \\ 
  transswitch0014r & --- & --- & --- & --- \\ 
  transswitch0030p & --- & --- & --- & --- \\ 
  transswitch0030r & --- & --- & --- & --- \\ 
  transswitch0039p & --- & --- & --- & --- \\ 
  transswitch0039r & --- & --- & --- & --- \\ 
  transswitch0057p & --- & --- & --- & --- \\ 
  transswitch0057r & --- & --- & --- & --- \\ 
  transswitch0118p & --- & --- & --- & --- \\ 
  transswitch0118r & --- & --- & --- & --- \\ 
  transswitch0300p & --- & --- & --- & --- \\ 
  transswitch0300r & --- & --- & --- & --- \\ 
  transswitch2383wpp & --- & --- &   1 &   0 \\ 
  transswitch2383wpr & --- & --- & --- & --- \\ 
  transswitch2736spp & --- & --- &   1 &   0 \\ 
  transswitch2736spr & --- & --- & --- & --- \\ 
  tspn05 & 191.25 &   0 &   1 &   1 \\ 
  tspn08 & 290.57 &   0 &   1 &   1 \\ 
  tspn10 & 225.13 &   0 &   1 &   1 \\ 
  tspn12 & 270.90 &   0 &   1 &   1 \\ 
  tspn15 & 334.73 &   2 &   3 &   6 \\ 
  unitcommit1 & 585256.00 &  14 &  35 &  58 \\ 
  unitcommit2 & 632283.00 &  44 &  23 &  90 \\ 
  uselinear & --- & --- & --- & --- \\ 
  util & 999.84 &   2 &  34 &  39 \\ 
  var\_con10 & 452.69 &   1 &   5 &   7 \\ 
  var\_con5 & 285.87 &   1 &   3 &   4 \\ 
  waste & 1169.08 &  18 &  82 & 139 \\ 
  wastepaper3 & 0.03 &   3 &  16 &  32 \\ 
  wastepaper4 & 0.03 &   2 &  22 &  32 \\ 
  wastepaper5 & 0.04 &   2 &  24 &  42 \\ 
  wastepaper6 & 0.01 &   2 &  19 &  37 \\ 
  water3 & --- & --- & --- & --- \\ 
  water4 & 1323.97 &  70 & 1294 & 1381 \\ 
  watercontamination0202 & 129.15 & 467 &  43 &  95 \\ 
  watercontamination0202r & 2690.24 &  25 &  97 & 196 \\ 
  watercontamination0303 & 580.69 & 811 &  60 & 137 \\ 
  watercontamination0303r & 8654.50 &  91 & 131 & 277 \\ 
  waterful2 & --- & --- & --- & --- \\ 
  waternd1 & --- & --- & --- & --- \\ 
  waternd2 & --- & --- & --- & --- \\ 
  waternd\_blacksburg & 518613.00 &  21 & 271 & 315 \\ 
  waternd\_fossiron & 346062.00 & 1035 & 14379 & 14507 \\ 
  waternd\_fosspoly0 & 93204600.00 &  10 &  61 & 126 \\ 
  waternd\_fosspoly1 & --- & --- &  10 & 33692 \\ 
  waternd\_hanoi & 7438010.00 &  36 & 525 & 640 \\ 
  waternd\_modena & 346062.00 & 1080 & 14379 & 14507 \\ 
  waternd\_pescara & 4575970.00 & 1060 & 11611 & 11743 \\ 
  waternd\_shamir & 434000.00 &   1 &  10 &  15 \\ 
  waterno1\_01 & 306.82 &  10 & 181 & 187 \\ 
  waterno1\_02 & 331.56 &   5 &  50 &  58 \\ 
  waterno1\_03 & 1257.91 &   9 & 106 & 117 \\ 
  waterno1\_04 & 1202.02 &  63 & 817 & 840 \\ 
  waterno1\_06 & 1093.12 &   7 &  33 &  54 \\ 
  waterno1\_09 & 1516.33 &  17 &  50 &  76 \\ 
  waterno1\_12 & 1826.62 &  35 & 140 & 166 \\ 
  waterno1\_18 & 2398.41 &  40 &  46 &  74 \\ 
  waterno1\_24 & 3130.12 & 1745 & 6004 & 6042 \\ 
  waterno2\_01 & 29.08 &   4 &  63 &  69 \\ 
  waterno2\_02 & 156.66 &   7 &  86 &  94 \\ 
  waterno2\_03 & 134.20 &   8 &  81 & 105 \\ 
  waterno2\_04 & 257.23 &  11 &  81 & 118 \\ 
  waterno2\_06 & 442.50 &  16 &  88 & 137 \\ 
  waterno2\_09 & 1422.16 &  31 &  87 & 159 \\ 
  waterno2\_12 & 3119.55 &  48 &  99 & 211 \\ 
  waterno2\_18 & 7359.11 & 105 & 107 & 253 \\ 
  waterno2\_24 & 9711.81 & 301 & 477 & 653 \\ 
  waters & --- & --- & --- & --- \\ 
  watersbp & --- & --- & --- & --- \\ 
  watersym1 & --- & --- & --- & --- \\ 
  watersym2 & --- & --- & --- & --- \\ 
  watertreatnd\_conc & --- & --- & --- & --- \\ 
  watertreatnd\_flow & 349525.00 &  19 & 294 & 297 \\ 
  waterx & 934.86 &   0 &   7 &  12 \\ 
  waterz & 315610.00 & 1575 & 13551 & 26704 \\ 
  windfac & 0.25 &   2 &  26 &  28 \\ 
  \end{longtable}


\end{document}